\documentclass[11pt]{amsart}
\usepackage[letterpaper,margin=1in]{geometry}
\usepackage{epsfig}
\usepackage{graphicx}
\usepackage{rotating}
\usepackage{color}
\usepackage[usenames,dvipsnames,svgnames,table]{xcolor}
\usepackage{bm}  
\usepackage{amssymb,amsmath,amsfonts}
\usepackage{amscd}
\usepackage{verbatim}
\usepackage{subfigure}

\definecolor{ludmilscolor}{RGB}{255,0,255}

\DeclareMathAlphabet\mathbold{OML}{cmm}{b}{it}

\definecolor{black}{rgb}{0,0,0}

\numberwithin{equation}{section}

\newtheorem{theorem}{Theorem}[section]
\newtheorem{lemma}{Lemma}[section]
\newtheorem{remark}{Remark}[section]

\newtheorem{definition}{Definition}[section]
\newtheorem{algorithm}{Algorithm}[section]

\theoremstyle{definition}\newtheorem{example}{Example}[section]

\renewcommand{\div}{\operatorname{div}}

\newcommand{\RT}[1]{{\mathcal{RT}}_{#1}} 
\newcommand{\BDM}[1]{{\mathcal{BDM}}_{#1}} 

\newcommand{\ddiv}{\operatorname{div}}
\newcommand{\Bdiv}{B_{\operatorname{div}}}

\def\b1{{\mathbf 1}}
\def\bv{{\mathbf v}}
\def\bu{{\mathbf u}}
\def\bw{{\mathbf w}}
\def\bb{{\mathbf b}}

\def\bn{{\mathbf n}}

\def\bH{{\boldsymbol H}}

\def\bV{{\boldsymbol V}}
\def\bVh{{\boldsymbol V}_{\hspace{-0.2mm}h}}
\def\bHN{{\boldsymbol H}_{\hspace{-0.2mm}N}}

\def\bbR{\mathbb{R}}

\def\Mal{{M}_\al}

\newcommand{\R}{{\mathcal R}}
\newcommand{\A}{{\mathcal A}}

\newcommand{\vek}[1]{{\bf{#1}}}
\newcommand{\Reals}[1]{{\bbR}^{#1}}

\newcommand{\GCG}{{\rm GCG,\nu}}

\def\al{{\alpha}}
\def\Om{{\Omega}}

\title[Preconditioning of weighted $\bH(\ddiv)$-norm and applications]
{Preconditioning of weighted $\bH(\ddiv)$-norm and applications to numerical simulation of
highly heterogeneous media}

\author[J.~Kraus, R.~Lazarov, M.~Lymbery, S.~Margenov, L.~Zikatanov]
{Johannes Kraus, Raytcho  Lazarov, Maria Lymbery, Svetozar Margenov, and Ludmil Zikatanov}

\address{Faculty of Mathematics, University of Duisburg-Essen, Thea-Leymann-Str. 9 45127 Essen, Germany}
\email{johannes.kraus@uni-due.de}

\address{Department of Mathematics,
Texas A \& M University, College Station, TX 77843, USA and Institute of Mathematics
and Informatics, Bulgarian Academy of Sciences, Acad. G. Bonchev St., Bl. 8,  1113 - Sofia, BULGARIA}
\email{lazarov@math.tamu.edu}

\address{Faculty of Mathematics, University of Duisburg-Essen, Thea-Leymann-Str. 9 45127 Essen, Germany}
\email{maria.lymbery@uni-due.de}

\address{Institute of Information and Communication Technologies, Bulgarian Academy of Sciences,
Acad. G. Bonchev St., Block 2,  1113 - Sofia, BULGARIA}
\email{margenov@parallel.bas.bg}

\address{Department of Mathematics,
The Pennsylvania State University, University Park, PA 16802, USA and Institute of Mathematics
and Informatics, Bulgarian Academy of Sciences, Acad. G. Bonchev St., Bl. 8,  1113 - Sofia, BULGARIA}
\email{ltz@math.psu.edu}



\date{February 26, 2013--beginning; Today is \today}
 \thanks{}

\begin{document}

\begin{abstract}
In this paper we propose and analyse  a preconditioner for a system arising from a mixed finite element
approximation of second order elliptic problems describing processes in
highly heterogeneous media. 
Our approach uses the technique of multilevel methods (see,
e.g. \cite{2008VassilevskiP-aa}) and the recently proposed
preconditioner based on additive Schur complement approximation by
J. Kraus \cite{Kraus_12}. The main results are  the design, study, 
and numerical justification of iterative algorithms for
such problems that are robust with respect to the contrast of the
media, defined as the ratio between the maximum and minimum values of
the coefficient of the problem. The numerical tests provide an
experimental evidence of the high quality of the preconditioner and its desired robustness
with respect to the material contrast. Numerical results for several 
representative cases are presented, one of which is 
related to the SPE10 (Society of Petroleum Engineers) benchmark problem.

\end{abstract}

\maketitle


\section{Introduction}\label{sec:intro}
\subsection{Model problem definition}\label{ssec:model}
Flows in porous media appear in many industrial, scientific, engineering and
environmental applications and are a subject of significant scientific interest. 
The same mathematical formulation is also used in modelling of 
other  physical processes such as heat and mass
transfer, diffusion of passive chemicals and electromagnetics. 
This leads to the following system of partial differential equations of first order 
for the unknown scalar functions $p(x)$ and the vector function $\bu(x)$:
\begin{subequations}\label{equation-1}
\begin{alignat}{2}
 \bu + K(x) \nabla p &= 0 \qquad && \text{in $\Omega$,}\\
\label{equation-2}
\ddiv \bu  &= f && \text{in $\Omega$},\\
\label{D BC}
p &= g && \text{on $\Gamma_D$ , } \\
\label{N BC}
\bu \cdot \bn  &=0 && \text{on $\Gamma_N$, }
\end{alignat}
\end{subequations}
where $\Omega$ is a polygonal  domain in $\mathbb{R}^d$, $d=2,3$. 
In the terminology of flows in porous media the unknown scalar functions $p(x)$ 
and the vector function $\bu$ are called pressure and velocity respectively, 
while $K(x): \mathbb{R}^d\mapsto \mathbb{R}^{d\times d}$, called 
the permeability tensor, is a symmetric and 
positive definite (SPD) matrix for almost all $x \in \Omega$. 
The first equation is the Darcy law and the second equation expresses conservation of mass.

Our study is focused on the case $K(x) = k(x) I$, where  
$I$ is the identity matrix in $ \mathbb{R}^d$  and $k(x)$ is a scalar function.
 The given forcing term $f $ is function in $ L^2(\Omega)$. 
The boundary $\partial \Omega $ is split into two
 non-overlapping parts $\Gamma_D$ and $\Gamma_N$ and in the case of a pure
 Neumann problem, i.e. $\Gamma_N=\partial \Omega$, we assume that $f$
 satisfies the compatibility condition $\int_\Omega f dx=0$. In such a
 case the solution is determined uniquely by taking $\int_\Omega p ~dx=0$.

To simplify the presentation, $\Gamma_D$ is assumed to be a
 non-empty set with strictly positive measure which is also closed
 with respect to $\partial \Omega$ and $g(x) \equiv 0$ on $\Gamma_D$, so the above system of equations
 has a unique solution $ p \in H_D^1(\Omega):=\{ q \in H^1(\Omega): ~ q =0 \,\, \mbox{on} \,\, \Gamma_D\}$.

Specifically, applications 
to  flows in {\it highly heterogeneous} porous media of {\it high contrast} are studied.
The coefficient $k(x)$ in this context represents media with multiscale features, involving
many small size inclusions and/or long connected subdomains (channels), where $k(x)$ has large 
values (see Figure \ref{fig:islands_random}). 
A computer generated permeability coefficient $K(x)$ which exhibits such
features has been used as a benchmark in petroleum engineering related simulations,
cf. SPE10 Project~\cite{SPE10_project}. 
In Figure \ref{fig:spe_10} is shown the permeability field of 2-dimensional slices of such media.
An important characteristic is the contrast $\kappa$,
defined by \eqref{contrast} as a ratio between the maximum and minimum values of  $k(x)$.

In this paper we consider  approximations of the problem
\eqref{equation-1} by the mixed finite element method on a mesh that resolves the finest scale of
the permeability. This leads to a very large indefinite symmetric system of algebraic equations. 
Developing, studying and testing an optimal preconditioner with respect to the contrast ${\kappa}$ and the mesh
size $h$  for this algebraic problem is the objective of this paper.  
Our considerations and numerical experiments 
show that the proposed preconditioner  
is optimal so that the number of iterations
depends neither on the contrast nor the mesh size.
This is the main achievement in this paper.

For the vector variable $\bu$ we use the lowest order Raviart-Thomas $\bH(\ddiv)$-conforming finite elements.  
The algebraic  system of linear equations
for the unknown degrees of freedom associated with 
$\bu$ and $p$ can be written in the following block form
(see also for more details Subsection~\ref{sec:matrix-notation}) 
\begin{equation}\label{eq:saddle_point_sys-intro}
 \left[
\begin{array}{cc}
\Mal & -\Bdiv^T \\
-\Bdiv & 0
\end{array}
\right]
\left[
\begin{array}{c}
 {\bf u} \\
 {\bf p}
\end{array}
\right]=
\left[
\begin{array}{c}
 \bf{0} \\
 \bf{f}
\end{array}
\right],
\end{equation}
where the matrix $\Mal$ is generated by the inner product $(\al \bu, \bv)$ while 
$B_{\ddiv} $ by the form $(\nabla \cdot \bu, q)$. 
It is well known (see, e.g. \cite{Arnold1997preconditioning}) that  the mapping properties 
of a matrix of this system are the same as those of 
\begin{equation}\label{eq:AFW_preconditioner_intro}
 \mathcal{B}_h:=\left[
\begin{array}{cc}
 A & 0\\[2ex]
 0 & I
\end{array}
\right],
\end{equation} 
where the matrix $A$ corresponds to the 
weighted $\bH(\ddiv)$-inner product $(\al \bu, \bv) + (\ddiv \bu, \ddiv \bv)$. 
Therefore, for an optimal MINRES iteration, the construction of an efficient 
preconditioner of the bilinear form  $(\al \bu, \bv) + (\ddiv \bu, \ddiv \bv)$ 
which is robust with respect to both the contrast and the mesh-size is essential. 
In this paper we focus on the construction and study of a suitable 
$A$ in \eqref{eq:AFW_preconditioner_intro}. 

\subsection{Overview of existing results}\label{ssec:existing}

The standard elliptic theory ensures the existence of 
a unique solution $p \in H^1_D(\Om)$. However, since the
coefficient matrix $K(x)$ is piece-wise smooth and may have very large
jumps, the solution $p$ has low regularity. For example, the case $H^{1+s}(\Om)$
where $s>0$ could depend on the contrast $\kappa$ in a subtle and
unfavourable manner. 
This must be taken into account when proving the stability of 
discrete methods with a  constant independent of
$\kappa$. As a consequence, any solution or preconditioning technique,
such as multigrid and domain decomposition that are analysed by 
using the solutions' regularity, cannot produce theoretical results
independent of the contrast.

Note that block $A$ corresponds to the finite element approximation of the
weighted $\bH(\ddiv)$-norm generated by the weighted inner product 
$(K^{-1}\bu, \bv) + (\div \bu, \div \bv)$ with
$\bH(\ddiv)$-conforming finite elements. Thus, one might expect that the existing preconditioners of 
$\bH(\ddiv)$-norms would be appropriate to begin with.  
Various  scenarios for the properties of $K(x)$ are possible.

{\it Constant $K$ and/or smooth variable $K(x)$.} 
The case of 
$K(x)$ being an SPD matrix over $\Omega$ has been considered by Arnold, Falk, and Winther in
\cite{Arnold1997preconditioning, Arnold2000MG} and the corresponding preconditioner (based on 
mutigrid and/or domain decomposition) was
shown to be optimal with respect to the mesh-size in two space dimensions. 
The analysis of the preconditioner relies on the approximation properties
of the Raviart-Thomas projection and requires full regularity of the solution.
Further, based on early work by Vassilevski and Wang 
\cite{VassilevskiWang1992}, Hiptmair \cite{Hiptmair_2001} and later Hiptmair and Xu \cite{hiptmair2007nodal}
developed a preconditioner for the $\bH(\ddiv)$-norm that is optimal with respect to the mesh-size.
This work does not consider the variable $K(x)$ and a weighted norm. Nevertheless 
its analysis can be potentially extended to this case. However, the theoretical justification 
of this preconditioner uses in a fundamental way the approximation properties of finite element 
projections  (Raviart-Thomas  in 2-D and Nedelec in 3-D) that require regularity of the 
vector field  $\bu$, see e.g. \cite[error bounds (2.3) -- (2.5)]{Hiptmair_2001}, which 
may depend in a unfavourable way on the contrast $\kappa$. 
In general, such regularity is not available for problems in highly heterogeneous media with 
large contrast.  Additionally, the main ingredient of the preconditioner in~\cite{hiptmair2007nodal}, 
stable regular decompositions,
requires an extension of these results to the case of weighted norms. To the best of our knowledge, such
results are still out of reach for highly heterogenous coefficients and the analogues of Lemma~3.4 and Lemma~3.8
from \cite{hiptmair2007nodal}, crucial for constructing preconditioners, are not yet available.

{\it Anisotropic coefficient matrix $K(x)$.}
Often in applications the coefficient matrix represents anisotropic media or heterogeneous media
with highly anisotropic inclusions. Solving linear systems resulting from finite element approximation
of such problems is not a fully mastered task. Moreover, a theoretical 
justification of iterative solvers that are robust with respect to anisotropy is a very
difficult task. Some of early work in multigrid 
for two-dimensional problems, see, e.g. \cite{Hiptmair_2001,Bramble_Zhang_aniso}, 
uses grids that are aligned with the anisotropy. In \cite[conditions (1.2) -- (1.4)]{Bramble_Zhang_aniso}, 
a certain  coefficient regularity is required, while in \cite{Hiptmair_2001} it is assumed that the 
grid is aligned with the anisotropy of the coefficient $K(x)$  and the  line-relaxation in the strongly coupled direction 
is combined with semi-coarsening in the other directions only.

{\it Variable $K(x)$ with high contrast $\kappa$.}
The design and analysis of condition numbers for
  the system \eqref{eq:saddle_point_sys-intro} preconditioned with
  block diagonal preconditioners were carried out by Powell and
  Silvester~\cite{powell2003optimal} and
  Powell~\cite{powell2005parameter}. One class of preconditioners
  proposed in these two papers, and relevant to our constructions here, is
  a block diagonal preconditioner with weighted $H(div)$ operator as
  one diagonal block, and lumped (weighted) mass matrix as the
  other. The practical preconditioners which result from this use
  approximations of the $H(div)$ problem with the $V$-cycle multigrid
  proposed by Arnold, Falk, and
  Winther~\cite{Arnold1997preconditioning}.  In case of a smooth
  coefficient matrix $K(x)$ such an approach produces an optimal
  preconditioner for the mixed system in the isotropic case,
  independent of the contrast $\kappa$ as seen in~\cite[Section~2.3.1,
  Tables~2.3--2.7]{powell2003optimal} and
  \cite[Tables~4,~5]{powell2005parameter}.  In other cases, matrix
  valued anisotropic coefficients and highly heterogenous coefficients
  not aligned with the grid, the resulting preconditioners are robust
  with respect to the coefficient variation whenever the approximation
  to the first diagonal block is. It is still an open theoretical
  question, however, whether any of the known multigrid algorithms for
  the weighted $H(div)$ problem converge uniformly with respect to
  both $h$ and the coefficient variation for all cases, and 
  particularly in cases of matrix valued and anisotropic coefficients,
  coefficients with discontinuities not aligned with the coarsest
  mesh.  This statement applies to \emph{both} algebraic and geometric
  multigrid methods.  Multigrid preconditioners such as Arnold, Falk,
  Winther MG and the Hiptmair Xu (HX) preconditioner, perform well in
  numerical tests and in some cases. However the mathematical theory
  confirming such numerical observations is still missing.

  Furthermore, the framework for practical preconditioning of
  Powell and Silvester~\cite{powell2005parameter} and
  Powell~\cite{powell2003optimal} can be combined the Schur complement
  preconditioning of the $H(div)$ block as proposed here.  While we
  are also missing aspects in the rigorous mathematical
  justification of such a result, the numerical tests presented later
  and the analysis in~\cite{powell2005parameter},
  and~\cite{powell2003optimal} show that such combined approach has a
  great potential for being successful and practical. 


{\it Highly heterogeneous discontinuous $K(x)$.}
In the existing literature there are a number of techniques for preconditioning algebraic problems 
with heterogeneous coefficients of high contrast or large jumps. 
Among the most popular are domain decomposition, e.g. \cite{klawonn2002dual, efendiev2012robust}, 
for the standard Galerkin FEM,
and multilevel methods for the hybridised mixed system, e.g. \cite{Kraus_12}. 
The main result in  \cite{klawonn2002dual} concerns a domain decomposition FETI-type 
preconditioner which is optimal with respect to the contrast in the case when 
jumps of $K(x)$ are aligned with the coarse mesh
(or the splitting of the domain into subdomains).
Similarly, the proposed preconditioners in \cite{Kraus2009robust},  based on algebraic 
multilevel methods (AMLI), are theoretically proven to be robust with respect to  
contrast in the case when jumps of the coefficients $K(x) $ are aligned 
with the initial coarse mesh. The results shown in~\cite[Table 7.10, p.~163]{Kraus2009robust}
demonstrate numerically that for highly heterogeneous media with jumps in the 
permeability $K(x)$ aligned with the fine mesh only, 
the AMLI-preconditioner is not robust with respect to the contrast. 
In a recent work,   \cite{Willems_SINUM_2014},  J. Willems has developed with respect to the problem parameters 
a robust nonlinear multilevel method for solving
general symmetric positive definite systems.  A crucial role in the construction of the 
nested spaces and the smoother is played by local generalised eigenproblems (in the style of \cite{galvis2010domain})
and  four assumptions. It is not clear when these are verifiable for the case of the 
form $\Lambda_\al$.

We share the opinion expressed in \cite[Section 6]{powell2005parameter}, 
that the existence of theoretically proven optimal precondtioners 
of the $\|\cdot \|_{\Lambda_\al}$-norm (defined by the weighted 
$\bH(\ddiv)$-product \eqref{WH-div-prod})  
in the case of a general symmetric and positive definite
tensor $K(x)$ is an open question. Moreover, a tensor $K(x)$ 
with arbitrary heterogeneities and/or high anisotropy represents a geniune challenge
for both theory and computational practice. 
Our paper is a step in this direction for the case  of a highly heterogeneous permeability
tensor $K(x)$ that satisfies the condition \eqref{bound_below_K}.

\subsection{Contributions of the study}\label{ssec:summary}

The main result and novelty of this paper is  the design, theoretical discussion and
experimental study of a preconditioner for a matrix corresponding to the 
weighted norm $\| \cdot\|_{\Lambda_\al}$  in the space $\bH(\ddiv)$ 
(defined by  \eqref{WH-div-prod}) which gives
an iterative method for mixed finite element systems 
converging independently of the contrast~$\kappa$.  Such a construction
is based on ideas from~\cite{Kraus_Lymb_Mar_2014}.

A crucial role is played by the well known {\it inf-sup} condition. In this paper we consider the case of 
a permeability tensor satisfying condition \eqref{bound_below_K}. 
The inf-sup condition for this case immediately follows from the well studied situation where $K(x)=I$.  
However,   in order to emphasise the dependence of the inf-sup constant 
on the global properties of the differential operator and the means by which it can be extended to 
a more general form of $K(x)$ an outline of the proof is presented.
Furthermore, we present a short discussion in which $K(x)=k(x)I$ for 
highly heterogeneous values of $k(x)$  and $ 0 < k(x) < \infty$. 
Theorem \ref{DM_stability} establishes an
inf-sup condition and boundedness of the corresponding bilinear form in a discrete setting. 
We emphasise that under \eqref{bound_below_K} the constant in the inf-sup condition and the
boundedness of the corresponding form does not depend on the contrast
of the media. 

In Section \ref{s:precon}, which is central to the paper, a new
preconditioning method for the finite element systems 
is described. 
Firstly, a block-diagonal preconditioner for the operator form of the mixed FEM is defined.
Furthermore, Subsection ~\ref{sec:matrix-notation}
presents the FE problem in matrix form. 
The key issue when designing a
contrast-independent Krylov solver is the construction of a
robust preconditioner for the weighted $H(\ddiv)$-norm. This is
addressed in detail in Subsection~\ref{sec:robust-H_div-precond} by
scrutinising an abstract auxiliary space two-grid method, see
Subsection~\ref{sec:as-2-grid-precond}.  A crucial role in our analysis
is played by Lemma \ref{lem:key}, where we establish key inequalities
regarding this two-grid preconditioner. In this analysis an important aspect
is the norm of a suitable projection characterising this preconditioner.  
Presently a general theoretical proof of the independence of this norm with respect 
to the contrast does not exist.
However, we have presented numerical evidence (in Section~\ref{sec:numerics},
Table~\ref{table:norm_p_d_tilda}) that this quantity is bounded
independently of the contrast and the mesh size. Such a result is highly
desirable and of great practical value. Moreover, all our presented numerical
experiments indirectly show such robustness. 
Subsection \ref{sec:asmg} introduces an
auxiliary space multigrid (ASMG) method. 
Two variants of the algorithm, differing only in their relaxation procedure, 
are described in Subsection~\ref{sec:algorithms}.  The work needed to compute the action of the preconditioner is 
proportional to the total number of non-zeroes in the coarse grid matrices (the so called operator complexity of the preconditioner), and this is 
discussed in some detail in Section 4.4.

Finally, Section~\ref{sec:numerics} gives numerical results for three different examples of porous media in two dimensions
in order to test the robustness of the preconditioner with respect to media contrast and its
optimality with respect to mesh-size. All numerical results confirm these claims.

\section{Problem formulation}\label{sec:problem}

\subsection{Notation and preliminaries}

For functions defined on $\Om$ we use the standard notations for
Sobolev spaces, namely, $H^s(\Om)$ for $s \ge 0$ being an integer is the
space of functions having their generalised derivatives up to order
$s$ square-integrable on $\Om$.  We denote by $(\cdot,\cdot)$ the
$L^2$ and $[L^2]^d$ inner products.
The standard norms on $H^s$ are denoted by $\|\cdot\|_{s}$.
For $s=0$ we often utilise $\|\cdot\|$ without a subscript.
When the
norm is weighted with a matrix valued function $\omega(x)$,
with $\omega(x)$ SPD for almost all $x\in \Omega$ we implement the
notation:
\begin{equation}\label{weighted-norms}
\|\bv\|_{0,\omega} :=
\|\omega^{1/2}\bv\|, \quad
|\phi|_{1,\omega}:= \|\nabla \phi\|_{0,\omega} =\|\omega^{1/2}\nabla \phi\|.
\end{equation} 
Occasionally, when considering only a subset of $\Omega$, e.g.
$T\subset \Omega$, then this is indicated in the notation for the norms
and seminorms, i.e., $\|\cdot\|_{s,T}$,
$\|\cdot\|_{s,\omega,T}$, $|\cdot|_{s,T}$ and $|\cdot|_{s,\omega,T}$.

To put our work in perspective, in the following, we consider 
$ K \in \mathbb{R}^{d\times d}$ to be a symmetric matrix,
the norm $\| K \|_{\ell^2}$ is, as usual, the spectral radius of
$ K$.
We define the number
\begin{equation}\label{contrast}
\kappa = \max_{x \in \Omega} (\|K(x)\|_{\ell^2}\|K^{-1}(x)\|_{\ell^2})
\quad \mbox{with} \quad
\|K(x)\|_{\ell^2} = \sup_{\xi \in \mathbb{R}^d}(K(x) \xi \cdot \xi)/(\xi \cdot \xi).
\end{equation}
to be the contrast of the media. 
Obviously, 
$\kappa = \max_{x \in \Om}K(x)/\min_{x \in \Om}K(x)$ 
for scalar permeability $K(x)$. In many applications
$\kappa$ could be many orders of magnitude, often up to $10$. 
Higher orders are of particular interest to us.

The Hilbert space $\bH(\ddiv)$ consists of square-integrable
vector-fields on $\Om$ with square-integrable divergence. 
The inner product in $\bH(\ddiv)$ is given by
\begin{equation}\label{H-div-prod}
\Lambda(\bu, \bv)= (\bu,\bv) + (\ddiv \bu, \ddiv \bv) 
\quad \mbox{and consequently} \quad 
\|\bv\|^2_{\bH(\ddiv)}:= \Lambda(\bv, \bv).
\end{equation}
Together with the Sobolev spaces $H^1_D(\Om)$ we use the following notation  for $\bH_N(\ddiv)$:
\begin{equation}\label{H-N}
  \bHN(\ddiv):= \bHN(\ddiv; \Om)= \{ \bv \in \bH(\div; \Om): ~~\bv(x) \cdot \bn=0 \quad \text{on} \quad \Gamma_N\}.
\end{equation}
Note that for $\phi\in H^1_D(\Om)$ the semi-norms
$|\phi|_1=\|\nabla \phi\|$ and
$|\phi|_{1,\omega}=\|\omega^{1/2}\nabla \phi\|$
are in fact norms on $H^1_D(\Omega)$ and we denote these by
$\|\phi\|_1$ and $\|\phi\|_{1,\omega}$.

Together with \eqref{H-div-prod} the following weighted inner product in the space $\bH(\ddiv)$
plays a fundamental role in our analysis
\begin{equation}\label{WH-div-prod}
\Lambda_\al(\bu, \bv)= (\al~\bu,\bv) + (\ddiv \bu, \ddiv \bv),
\quad \alpha(x) = K^{-1}(x),
\end{equation}
which defines the norm
\begin{equation}\label{norms}
\|\bv\|^2_{\Lambda_\alpha}  =
\Lambda_\al(\bv, \bv)=
\|\bv\|^2_{0,\alpha} + \|\ddiv \bv\|^2.
\end{equation}
Note, that a weighted bilinear form of the type $
\Lambda_{\alpha,\beta}(\bu, \bv)= \alpha (\bu, \bv) + \beta (\ddiv
\bu, \ddiv \bv) $ with $\alpha > 0$ and $\beta > 0$ constants was
used by Arnold, Falk, and Winther in \cite{Arnold2000MG} to design
multigrid methods for $\bH(\ddiv)$-systems.  A key point in their
study was the construction of a multigrid method that converges
uniformly with respect to the parameters $\alpha$ and $\beta$. The
important difference between our bilinear form $\Lambda_\alpha$
compared with $ \Lambda_{\alpha,\beta}$ is that in our scheme $\alpha$
is a highly heterogeneous function with high contrast.
This makes the proof of an {\it inf-sup} condition with the 
weighted $\bH(\ddiv)$-norm more delicate
and the construction of an efficient preconditioner a challenging task,
see Remark \ref{weighted_Poincare}. 

Note that for  all $\xi\in\mathbb{R}^d$, $x\in \Omega$ and  
$\|K\|_{\ell^2} = \sup_{\xi \in \mathbb{R}^d}(K \xi \cdot \xi)/(\xi \cdot \xi) $
we have with $K:=K(x)$
\begin{equation*}
(K\xi\cdot\xi) \ge 
(\xi\cdot\xi) \, \inf_{\theta\in \mathbb{R}^d}\frac{(K\theta\cdot \theta)}{(\theta\cdot\theta)} =
(\xi\cdot\xi) \, \inf_{\theta\in  \mathbb{R}^d}\frac{(\theta\cdot\theta)}{(K^{-1}\theta\cdot\theta)}
= \frac{(\xi\cdot\xi) }{\|K^{-1}\|_{\ell^2}}.
\end{equation*}

Throughout the paper  the following inequality is assumed
\begin{equation}\label{bound_below_K}
1  \le \min_{x \in \Omega} \|K(x)\|_{\ell^2} \quad \text{which implies}  \quad
(\xi\cdot \xi) \le  (K(x)\xi\cdot\xi), \quad \xi\in \mathbb{R}^d. 
\end{equation} 
As seen from the considerations above, such an assumption is 
fulfilled if we scale the coefficient 
\[
K(x)\leftarrow
K(x)\max_{x \in \Omega}  \|K^{-1}(x)\|_{\ell^2}.
\]
Clearly such rescaling does not change the value of $\kappa$.
However, it would change the right hand side 
\[
f(x)\leftarrow
f(x)\max_{x \in \Omega} \|K^{-1}(x)\|_{\ell^2}
\]
and in general the stability of the solution cannot be established uniformly with respect to the 
contrast. 
Nevertheless, for homogeneous equations $ f(x) \equiv 0$ represents a large
class of applied problems. Such scaling then can be justified when the permeability is homogeneous near the Dirichlet boundary. 
Another possible case is when $f(x)=0$ in areas where
the permeability is very high.
The numerical examples 
of Powell and Silvester presented in \cite[Tables 2.9, 2.10]{powell2003optimal} for
$K(x)=k(x) I$, with $k(x)$ a scalar function and $I$ the identity matrix in $ \mathbb{R}^2$,
clearly show  this. The numerical experiments in Section~\ref{sec:numerics} 
consider homogeneous  equations. Furthermore, these 
problems are relevant to various numerical reservoir simulations. 

The case $0 < k(x) < \infty$, 
which could be used to model flow models in perforated domains,
appears to be more complicated and less studied. 
For such problems more advanced 
techniques involving weighted $L^2$-norms and the weighted Poincar\'e inequality are needed,
see Remark \ref{rem:Poincare}.  Such an inequality can be established
under certain restrictions on arrangement of the jumps 
and the topology of the permeability distribution, 
e.g. \cite{Rob_DD_XIX, Rob_Poincare_2013,Rob_Luso_Panayot_2012multilevel}.
Even more difficult is the case of tensor permeability $0 < \|K(x)||_{\ell^2} < \infty$,
which also includes models of flows in anisotropic highly heterogeneous media. 
These cases represent open problems with a wide range of 
applications and are left for further consideration and future studies.

\subsection{Weak formulations of the elliptic problem}\label{s:weak}
To present the dual mixed weak form we require the following notation, namely, 
$
\bV \equiv \bHN(\ddiv;\Omega)
$
and
$
W \equiv  L^2(\Omega). 
$

We multiply the first equation by $\al(x)=K^{-1}(x)$ and
a test function $ \bv$,  integrate over $\Om$, and perform integration by parts to obtain
\begin{equation}\label{first}
(\al(x) \bu, \bv) - (p, \ddiv \bv) = 0
\end{equation}
Next, multiplying the second equation by a test function
$q$ and integrating over $\Om$ gives
\begin{equation}\label{div-eqn}
(\ddiv \bu, q)=(f,q).
\end{equation}
Then the weak form of  the problem \eqref{equation-1} is: find  $\bu \in \bV$ and $p \in W$ such that
\begin{equation}\label{eq:dual_mixed}
\A^{}(\bu,p; \bv,q)= - (f, q),
\quad \mbox{for all}\quad (\bv, q) \in \bV \times W,
\end{equation}
where the bilinear form
$\A^{}(\bu,p; \bv,q): (\bV, W)  \times ( \bV, W) \to \bbR$ is defined as
\begin{equation}\label{A-form-mixed}
\A^{}(\bu,p; \bv,q):=(\al \bu, \bv)
-(p, \ddiv \bv) - (\ddiv \bu, q). 
\end{equation}

\subsection{Stability of the weak formulations}
Consider the stability of the discrete problem \eqref{eq:dual_mixed}.
We use the Poincar\'e inequality
\begin{equation}\label{Poincare}
\text{there is $C_P>0$ such that} ~~\mbox{for all}\quad
q \in H^1_D(\Omega): \quad
\|q\|^2 \le C_P \| \nabla q\|^2. 
\end{equation}
The constant $C_P$ depends only on the geometry of the domain
$\Omega$ and the splitting of $\partial \Omega$ into $\Gamma_D$ and
$\Gamma_N$. Moreover, due to \eqref{bound_below_K} for the
coefficient $K(x)$ we also have the inequality
$$
\|q\|^2 \le C_P \| \nabla q\|^2 \le
C_P \|\nabla q\|^2_{0,K}.
$$

To show the stability of the weak formulation we need a continuity  and an 
{\it inf-sup} condition (see, e.g. \cite{1991BrezziF_FortinM-aa,Ern-Guermond})
for the bilinear form  $\A^{}(\bu,p; \bv,q)$
on the spaces $ \bV $ and $L^2(\Omega)$  equipped with a weighted norm
$ (\Lambda_\al(\bv,\bv))^{\frac{1}{2}}$ and standard $L^2$-norm $\|p\|$,
respectively.

\begin{lemma} \label{lem:stab_var}
Let $W=L^2(\Omega)$, $\bV=\bHN(\ddiv)$, and $\Vert \bv \Vert_{\Lambda_\al}:=
(\Lambda_\al(\bv,\bv))^{\frac{1}{2}}$. Assume also that the permeability  
coefficient $K(x)$ satisfies the inequality \eqref{bound_below_K}.
Then the following inequalities hold
\begin{enumerate}
\item 
For all $\bu, \bv\in\bV$ and  for all  $p, q\in W$
\begin{equation}\label{continuity}
\A^{}(\bu,p; \bv,q)
\le ( \|\bu\|^2_{\Lambda_\al} +\|p\|^2)^\frac12( \|\bv\|^2_{\Lambda_\al} +\|q\|^2)^\frac12;
\end{equation}
\item There is a constant $\al_0>0$ independent of $\alpha$ such that
\begin{equation}\label{A-infsup}
\sup_{\bv \in \bV, \, q \in W}
\frac{\A^{}(\bu,p; \bv,q) }{(\|\bv\|^2_{\Lambda_\al} +\|q\|^2)^\frac12}
\ge \alpha_0 ( \|\bu\|^2_{\Lambda_\al} +\|p\|^2 )^\frac12
\end{equation}

\end{enumerate}

\end{lemma}
\begin{proof}
The first inequality follows immediately by applying the Schwarz inequality to all three terms and keeping
in mind that $\al $ is a positive function. Proving the {\it inf-sup} condition \eqref{A-infsup} is equivalent
to proving the following inequality (see, \cite{Ern-Guermond}):
 \begin{equation}\label{BR_inf_sup}
  \inf_{q\in W} \sup_{\bv\in{\bV}}\frac{(\nabla\cdot \bv,q)}{\Vert \bv \Vert_{\Lambda_\al}
\Vert q \Vert}\ge
\gamma >0, \qquad \mbox{for all}\quad \bv\in\bV,\quad \mbox{for all}\quad q\in W.
 \end{equation}
 As is well known, if $\gamma $ is independent of the contrast
 $\kappa$, then so is $\alpha_0$. For more details on the relation
 between the constants $\gamma$ and $\alpha_0$ we refer to
 \cite{Xu2003Zikatanov}, see also Remark~\ref{rem:Poincare}. 
Furthermore, due to assumption  \eqref{bound_below_K} we have
$$
\Vert \bv \Vert_{\Lambda_\al} \le \|\bv\|_{\bH(\ddiv)} \quad \mbox{so that} \quad
  \inf_{q\in W} \sup_{\bv\in{\bV}}\frac{(\nabla\cdot \bv,q)}{\Vert \bv \Vert_{\Lambda_\al}
\Vert q \Vert} \ge   \inf_{q\in W} \sup_{\bv\in{\bV}}\frac{(\nabla\cdot \bv,q)}{\Vert \bv \Vert_{\bH(\ddiv)}
\Vert q \Vert}\ge \gamma >0.
$$
To find a computable bound for the constant $\gamma$ we can use the standard 
construction \cite{1991BrezziF_FortinM-aa} for the case $K(x)=1$ in $\Omega$.
For $q\in W$ we take $\bw = \nabla \varphi \in\bV$, where $\varphi\in H_D^1(\Omega)$ 
is the solution to the variational problem
$( \nabla \varphi,\nabla\chi ) = (q, \chi),$ for all $ \chi\in H_D^1(\Omega)$.
$\bw \in \bV$
with $\ddiv \bw= -q$ which holds in $L^2(\Omega)$ by construction; then using the above Poincar\'e inequality
we get
$\|\bw \| \le \sqrt{C_P} \|q\|$ so that,
$$ 
\sup_{\bv\in \bV} \frac{(q,\ddiv \bv)}{\|\bv\|_{\bH(\ddiv)}} \ge 
\frac{(q,\ddiv \bw)}{\|\bw\|_{\bH(\ddiv)}}=
\frac{\|q\|^2}{ (\|\bw\|^2+\|\ddiv \bw\|^2)^\frac12}
\ge \frac{\|q\|}{\sqrt{C_P+1}}.
$$

This shows \eqref{BR_inf_sup} with $\gamma =1/\sqrt{C_P+1}$, where $C_P$ is the constant in the
Poincar\'e inequality \eqref{Poincare}.
Then using the results of \cite{Xu2003Zikatanov,Krendl12stability} and inequalities \eqref{continuity} 
and \eqref{BR_inf_sup}
we deduce that 
the constant $\al_0$ in \eqref{A-infsup} is bounded from below. A sharp lower bound for 
$\al_0$ can be
obtained using the best known results of~\cite[Theorem 1]{Krendl12stability} to get
$
\al_0 \ge 1/( 2+ C_P),  
$
which completes the proof.  
\end{proof}

\begin{remark}\label{rem:Poincare}
As mentioned above, the case of scalar permeability $K(x)$,
$0< K(x) < \infty$, needs a different computational approach.
First we establish a special Poincar\'e inequality (involving the weighted $L^2$-norm)
with  a constant $C_P >0$
\begin{equation}\label{weighted_Poincare}
\| q \|^2_{0,K} := \int_\Omega K(x) q^2 \, dx \le C_P \| \nabla q\|_{0,K}^2  \quad
\mbox{where} \quad   \| \nabla q\|_{0,K}^2 = \int_\Omega K(x) |\nabla q|^2 \, dx.
\end{equation}
This type of inequality plays a role in domain decomposition methods, multiscale FEM
and multigrid preconditionters and has been studied
in e.g.~\cite{dryja1996multilevel,galvis2010domain,Rob_DD_XIX, Rob_Poincare_2013,Rob_Luso_Panayot_2012multilevel}.
Particularly relevant to our work is the study conducted in \cite{Rob_DD_XIX, Rob_Poincare_2013,Rob_Luso_Panayot_2012multilevel}
where, under certain 
restrictions on the distribution of the permeability $K(x)$, the constant $C_P$ in
\eqref{weighted_Poincare} is shown to be 
independent of the contrast~$\kappa$. 
Then using  \eqref{weighted_Poincare} one can prove the following {\it inf-sup} condition
 \begin{equation}\label{inf_sup_stab_var_weight}
  \inf_{q\in W} \sup_{\bv\in{\bV}}\frac{(\nabla\cdot \bv,q)}{(\Vert \bv \Vert^2_{0,\al} 
  + \Vert \nabla \cdot \bv \Vert^2_{0,\al})^{1/2}
\Vert q \Vert_{0,K}}\ge  \frac{1}{\sqrt{C_P+1}}
\qquad \mbox{for all}\quad \bv\in\bV,\quad q\in W.
 \end{equation}
However, this approach needs additional research for preconditioning the
weighted $H(div)$-norm 
$ (\Vert \bv \Vert^2_{0,\al}  + \Vert \nabla \cdot \bv \Vert^2_{0,\al})^{1/2}  $
and is left for future consideration.
\end{remark}

\section{FEM approximations}\label{s:FEM}

\subsection{Finite element partitioning and spaces}\label{ss:FE_spaces}
We assume that the domain $\Omega$ is connected and is triangulated
with $d$ dimensional simplexes. The triangulation is denoted by
$\mathcal{T}_h$ with the simplexes forming $\mathcal{T}_h$ assumed to be shape regular 
(the ratio between the diameter of a simplex and the inscribed ball is bounded above).
Now we consider the finite element approximation of problem~\eqref{equation-1} 
using the finite dimensional spaces  $\bVh \subset \bV$ 
and $W_h \subset W$ of piece-wise polynomial functions. 

It is well known that for the vector variable $\bu$ we can use $\bH(\ddiv)$-conforming
or Raviart-Thomas space $\RT{k}$ or Brezzi-Douglas-Marini $\BDM{k+1}$ finite elements.
However, since the problem has low regularity it is natural to use lowest order
finite element spaces. For the vector variable $\bu$ we use the standard
Raviart-Thomas $\text{RT}_0$ for simplexes and cubes. In the case of simplexes we can also apply
Brezzi-Douglas-Marini  $\text{BDM}_1$  finite elements.
Since $W$ is essentially $L^2(\Omega)$ 
for its finite element counterpart we can use a piece-wise constant function over the partition 
$\mathcal{T}_h$. We show that the corresponding finite element method is uniformly stable with
respect to the contrast $\kappa$.

\subsection{Stability of the mixed FEM} 
%
Thus, 
we take
\begin{equation}\label{space Vh}
\bVh=\{ \bv \in \bV: \, \bv |_T \in \RT{0} \,\,\, \mbox{for} \,\, T \in {\mathcal T}_h\}
\end{equation}
and
\begin{equation}\label{space Wh}
W_h =\{ q \in L^2(\Omega): \, q|_{T} \in {\mathcal P}_0,
\, \text{i.e. $q$ is a piece-wise constant function on} \,\,  {\mathcal T}_h\}.
\end{equation}
The mixed finite element approximation
of  the problem~\eqref{equation-1} 
 is: find  $\bu_h \in \bVh$ and $p_h \in W_h$ such that
\begin{equation}\label{eq:dual_mixed_FEM}
\A^{}(\bu_h,p_h; \bv,q)= - (f,q),
\quad \mbox{for all}\quad (\bv, q) \in \bVh \times W_h,
\end{equation}
where the bilinear form $ \A^{}(\bu_h,p_h; \bv,q)$ is defined by
\eqref{A-form-mixed}.
Our goal is to establish a discrete variant of the {\it inf-sup} condition.
\begin{lemma}\label{lem:stab_fin}
Let $\bVh$ be the space defined by \eqref{space Vh} and $W_h$ be the space defined by
\eqref{space Wh}. Assume also that the permeability  
coefficient $K(x)$ satisfies inequality \eqref{bound_below_K}. 
Then independent of the contrast $\kappa$ and the step-size $h$
the following inequality holds true:
\begin{equation}\label{inf_sup_fin_set}
 \inf_{q_h\in W_h}\sup_{\bv_h\in\bVh}
\frac{(\div \bv_h,q_h)}{\Vert \bv_h\Vert_{\Lambda_\al} \Vert q_h\Vert}\ge
\gamma >0 . 
\end{equation}
\end{lemma}
\begin{proof}
First we note that {\it inf-sup} condition for the case $K(x)=1$ 
is well known, \cite{1991BrezziF_FortinM-aa,Ern-Guermond}.
Then using the same argument as in the proof 
of Lemma \ref{lem:stab_var} we show the desired result. 
Note that the constant $\gamma$ will depend on the constant $C_P$ of the Poinacar\'e 
inequality and  the properties of the 
finite element partitioning, but is not dependent on the contrast $\kappa$. 
\end{proof}
As a consequence of Lemma \ref{lem:stab_fin} and \eqref{continuity}  we have
\begin{theorem}\label{DM_stability}
Assume that the permeability  
coefficient $K(x)$ satisfies the inequality \eqref{bound_below_K}. 
Then the following bounds are valid for all $\bu \in \bVh$ and $ p \in W_h$:
\begin{equation}\label{DM-Fem_stability}
\alpha_0 ( \|\bu\|^2_{\Lambda_\al} +\|p\|^2)^\frac12 \le \sup_{\bv \in\bVh, q\in W_h}
\frac{\A^{}(\bu,p; \bv,q) }{(\|\bv\|^2_{\Lambda_\al} +\|q\|^2)^\frac12} \le 
( \|\bu\|^2_{\Lambda_\al} +\|p\|^2)^\frac12.
\end{equation}
The constant $\alpha_0 > 0$ may depend on the shape regularity of the
mesh. However, it is  independent of the contrast $\kappa$ and the mesh-size $h$.
In fact, $\al_0 \ge 1/(1+1/\gamma^2)$, where $\gamma$ is the constant in \eqref{inf_sup_fin_set}.
\end{theorem}

\section{Preconditioning}\label{s:precon}

\subsection{Block-diagonal preconditioner for the system of the finite element method}\label{sec:b-diag-operator}
Now we consider problem \eqref{eq:dual_mixed} and for definiteness
we restrict ourselves to lowest order Raviart-Thomas mixed finite
elements on a rectangular grid. The goal of this section is to develop
and justify a preconditioner for the algebraic problem resulting from
the Galerkin method \eqref{eq:dual_mixed} that is independent of the media contrast.

Then \eqref{eq:dual_mixed} can be written as an operator equation in the space $X_h=\bV_h\times W_h$
equipped with the norm $\|{\bm x}_h \|^2_{X_h}= \| \bu_h\|^2_{\Lambda_\al} + \|p_h\|^2$   
for ${\bm x}_h =(\bu_h, p_h)$. Then,
\begin{equation}\label{mixed-operator}
 \mathcal{A}_h {\bm x}_h = {\bm f}_h, \quad \mbox{for} \quad
{\bm f}_h =({\bm 0}, - f_h) \in X_h,
\end{equation}
where for all
$\bm{y}_h=(\bv_h,q_h)\in X_h$
$$
\langle \mathcal{A}_h {\bm x}_h, {\bm y}_h \rangle = \A^{}(\bu_h,p_h; \bv_h,q_h).
\quad
$$
Here $\langle \cdot, \cdot \rangle$ denotes the duality between $  X_h^{\star}$ and $X_h$.
Obviously, the operator $\mathcal{A}_h:X_h\rightarrow X_h^{\star}$ is self-adjoint
on $X_h=\bV_h\times W_h$ and indefinite.

Now the right inequality in \eqref{DM-Fem_stability} 
can be written as 
$$
\| \mathcal{A}_h {\bm x}_h\|_{X_h^{\star}} =  
\sup_{{\bm y}_h \in X_h} 
\frac{ \langle \mathcal{A}_h {\bm x}_h, {\bm y}_h \rangle } {\| {\bm y}_h\|_{X_h}}
\le c \|{\bm x}_h \|_{X_h},
$$
where $c=1$. 
This means that $  \Vert\mathcal{A}_h\Vert_{\mathcal{L}(X_h,X_h^{\star})} \le c$.
Similarly, the left inequality in \eqref{DM-Fem_stability} 
leads to $\Vert\mathcal{A}_h^{-1}\Vert_{\mathcal{L}(X_h^{\star},X_h)} \le c $,
where $c=1/{\alpha_0}$.

Now our goal is to construct a positive definite self-adjoint operator
$\mathcal{B}_h:X_h\rightarrow X_h^{\star}$ such that all eigenvalues of
$\mathcal{B}_h^{-1}\mathcal{A}_h$ are uniformly bounded 
independent of $h$ and, more importantly, independent of the
contrast $\kappa$. 
Already, since 
$  \Vert\mathcal{A}_h\Vert_{\mathcal{L}(X_h,X_h^{\star})} \le c$ and 
$\Vert\mathcal{A}_h^{-1}\Vert_{\mathcal{L}(X_h^{\star},X_h)} \le c$ 
with $c$ independent of the contrast, 
then it follows that
\begin{equation}\label{eq:precond_norms}
 \Vert\mathcal{B}_h\Vert_{\mathcal{L}(X_h,X_h^{\star})}\;\; \mbox{and}\;\;
\Vert\mathcal{B}_h^{-1}\Vert_{\mathcal{L}(X_h^{\star},X_h)} \;\; \mbox{being uniformly bounded in} \;\; h \;\; \mbox{and}
\;\; \kappa
\end{equation}
is sufficient for $\mathcal{B}_h$ to be a uniform and robust
preconditioner for the minimum residual (MinRes) iteration.

Define the block-diagonal matrix 
\begin{equation}\label{eq:AFW_preconditioner}
 \mathcal{B}_h:=\left[
\begin{array}{cc}
 A_h & 0\\[2ex]
 0 & I_h
\end{array}
\right], 
\end{equation}
where $A_h:\,\bV_h\rightarrow \bV_h^{*}$ is given by
$
(A_h{\bu}_h,{\bv}_h):=\Lambda_\al(\bu_h,\bv_h)=
(\al \, {\bu}_h,{\bv}_h)+(\nabla\cdot {\bu}_h,\nabla\cdot {\bv}_h)
$ 
and $I_h$ is the identity operator in $W_h$.
Then estimates of the eigenvalues of $\mathcal{B}_h^{-1}\mathcal{A}_h$ are obtained in a standard manner:
consider the corresponding algebraic problem of finding the eigenpairs $(\lambda, {\bm x}_h)$, 
$\mathcal{A}_h {\bm x}_h = \lambda \mathcal{B}_h {\bm x}_h$, and use the above 
properties of  $\mathcal{A}_h$ and $ \mathcal{B}_h$, 
for more details, see \cite{powell2005parameter,powell2003optimal,vassilevski1996preconditioning}.

Then condition~\eqref{eq:precond_norms} reduces to $\Vert A_h
\Vert_{\mathcal{L}(\bV_h,\bV_h^{\star})}$ and $\Vert
A_h^{-1}\Vert_{\mathcal{L}(\bV_h^{\star},\bV_h)}$ being uniformly
bounded in $h$ and $\kappa$, which is
sufficient for optimality of the preconditioner, ~\cite{Arnold1997preconditioning}.  
Thus, the main task  in this section is the development and study of
a robust and uniformly convergent, with respect to $h$ and $\kappa$,
iterative method for solving systems with
$A_h \bu_h=\bb_h$.

\begin{remark}
We note that any successful development of a robust preconditioner $A_h$ could be also
used in the least-squares approximation of this problem written in a mixed form. 
In the least-squares approximation, e.g. \cite{pehlivanov1994least}, 
the upper right block is essentially an operator 
generated by the weighted $\bH(\ddiv)$-inner product 
$(\al \bu_h,\bv_h) +(\nabla \cdot \bu_h, \nabla \cdot \bv_h)$.
\end{remark}

\subsection{Reformulation of the FE problem using matrix notation}\label{sec:matrix-notation}
The derivation and the justification of the preconditioner are in the
framework of algebraic multilevel/multigrid methods. As a first step
we rewrite the operator equation \eqref{mixed-operator} in a matrix
form. Instead of functions ${\bm x}_h =(\bu_h, p_h)
\in \bV_h \times W_h$ we use vectors consisting of the
degrees of freedom determining ${\bm x}_h$ 
through the nodal basis functions, namely,
$$
{\bm x}=\left[
\begin{array}{c}
 {\bf u} \\
 {\bf p}
\end{array} \right ],
\quad \mbox{where} \quad
\quad {\bf u} \in \mathbb{R}^{|\mathcal{E}_h|},
\quad {\bf p} \in \mathbb{R}^{|\mathcal{T}_h|},
\quad \mbox{are vector columns}
$$
and $ {|\mathcal{E}_h|}$ is the number of edges in $ \mathcal{E}_h$, excluding those
on $\Gamma_N$, and $ {|\mathcal{T}_h|}$  is the number of
rectangles of the partition $ \mathcal{T}_h$.  Then $A$, $\Bdiv$,
$\widetilde A$ and $R$, denote matrices being either square or
rectangular.  As a result of this convention,
\eqref{mixed-operator} can be written in a matrix form~\eqref{eq:saddle_point_sys-intro}.
Our aim now is to derive and study a preconditioner for algebraic systems of the form \eqref{eq:saddle_point_sys-intro}, 
which due to the above considerations reduces to the efficient preconditioning of the system
\begin{equation}\label{algebraic-hdiv}
A \bu = \bb, \quad \bu, \bb \in  \mathbb{R}^N, \quad N:={|\mathcal{E}_h|}.
\end{equation}

\subsection{Robust preconditioning of the weighted $\bH(\ddiv)$-norm}\label{sec:robust-H_div-precond}

\cite{Kraus_12} has introduced the additive Schur complement approximation (ASCA) as a tool for constructing 
robust coarse spaces for high-frequency high-contrast problems.  Recently, this technique has
also been utilised as a building block for a new class of multigrid
methods in which a coarse-grid correction, as used in standard multigrid
algorithms, is replaced by an auxiliary-space
correction~\cite{Kraus_Lymb_Mar_2014}.  Viewed as a block
factorisation algorithm, the major computations in this so-called
auxiliary space multigrid (ASMG) method can be performed in parallel
since they consist of a two-level block factorisation of local
finite element stiffness matrices associated with a partitioning of
the domain into overlapping or non-overlapping subdomains. The
analysis of the two-grid ASMG preconditioner relies on the fictitious
space lemma, see~\cite{Kraus_Lymb_Mar_2014}. However, the underlying 
construction is purely algebraic and thus essentially differs from the
methodology in~\cite{hiptmair2007nodal}.

In this section we recall the basic construction of the
ASMG-method and specify modifications that allow its
successful application the linear systems arising from
$\bH(\ddiv)$-conforming discretisations of the subproblem
involving the weighted $\bH(\ddiv)$ bilinear form~\eqref{WH-div-prod}.

\subsubsection{Additive Schur complement approximation}\label{sec:asca}

The first step in the construction of the preconditioner involves a covering of the domain
$\Omega$ by $n$  overlapping subdomains $\Omega_{i}$, i.e.,
$ 
 \overline{\Omega}=\bigcup_{i=1}^{n} \overline{\Omega}_{i}.
$ 
This overlapping covering of $\Omega$ is to some extent
arbitrary with generous overlap.
For practical purposes, however, we consider the case in Figure \ref{fig:covering}.
For this any finite element in the partition $\mathcal{T}_h$
belongs to no more than four subdomains.
We associate subdomain matrices $A_{i}$,
$i=1,\dots,n$ with the subdomains $\Omega_{i}$, corresponding to the degrees of freedom
in domain $\Omega_i$, and assume that $A$ is assembled via
$$
A=\sum_{i=1}^{n} R_{{i}}^T A_{{i}} R_{{i}},
$$
where $ R_{i}$ is a rectangular matrix extending by zero the vector associated with the
degrees of freedom of $\bu_h$ in $\Omega_i$ to a vector representing the degrees of
freedom in the whole domain $\Omega$.
Assume further that the set $\mathcal{D}$ of degrees of freedom (DOF) of $\bu_h$
is partitioned into a set $\mathcal{D}_{\rm f}$, fine DOF, and a set
$\mathcal{D}_{\rm c}$, coarse DOF, so that
\begin{equation}\label{eq:partitioning_DOF}
\mathcal{D} = \mathcal{D}_{\rm f} \oplus \mathcal{D}_{\rm c},
\end{equation}
where $N_1:=\vert\mathcal{D}_{\rm f} \vert$ and $N_2:=\vert\mathcal{D}_{\rm c} \vert$
denote the cardinalities of $\mathcal{D}_{\rm f}$ and $\mathcal{D}_{\rm c}$, respectively,
with $N_1 + N_2 =N:=|{\mathcal E}_h|$. Recall that  $ |{\mathcal E}_h|$
is the number of edges in the partitioning ${\mathcal T}_h$ with the
edges on $\Gamma_N$ excluded.
Such a splitting is not obvious for the mixed finite element
method and is explained in detail later.
The splitting
\eqref{eq:partitioning_DOF} induces a representation of the matrices $A$ and
$A_{i}$ in two-by-two block form, i.e.,
\begin{equation}\label{eq:A_two_by_two-x}
A=\left[
        \begin{array}{cc}
        A_{11} & A_{12} \\
        A_{21} & A_{22}
        \end{array}
\right],\qquad
 A_{i}=\left[
        \begin{array}{cc}
        A_{i:11} & A_{i:12} \\
        A_{i:21} & A_{i:22}
        \end{array}
\right], \quad i=1,\dots,n.
\end{equation}
We now introduce the following auxiliary domain decomposition matrix
\begin{equation}\label{eq:auxA}
\widetilde{A}=
\left[ \begin{array}{ccccc}
A_{1:11} &&&& A_{1:12} R_{1:2} \\[0.5ex]
& A_{2:11} &&& A_{2:12} R_{2:2} \\[0.5ex]
&& \ddots && \vdots \\[0.5ex]
&&& A_{{n}:11} & A_{{n}:12} R_{n:2} \\[0.5ex]
R_{1:2}^T A_{1:21} & R_{2:2}^T A_{2:21} & \hdots & R^T_{n:2} A_{{n}:21}
& \displaystyle \sum_{i=1}^{n} R^T_{i:2} A_{i:22} R_{i:2}
                \end{array}
\right] .
\end{equation}
Setting
$\widetilde{A}_{11}={\rm diag}\{A_{1:11}, \ldots, A_{{n}:11}\}$,
$\widetilde{A}_{22}=\sum_{i=1}^{n} R^T_{i:2} A_{i:22} R_{i:2}$ we have
\begin{equation}\label{eq:A_two_by_two}
\widetilde{A}=\left[
        \begin{array}{cc}
        \widetilde{A}_{11} & \widetilde{A}_{12} \\
        \widetilde{A}_{21} & \widetilde{A}_{22}
        \end{array}
\right].
\end{equation}
Note that if $A$ is an SPD matrix, it follows that $\widetilde{A}$ is a symmetric and positive semi-definite matrix.
Moreover, $A \in \Reals{N{\times}N}$ and
$\widetilde{A} \in \Reals{\widetilde{N}{\times}\widetilde{N}}$
are related via
\begin{equation}
A=R \widetilde{A} R^T
\end{equation}
where
\begin{equation}\label{eq:R}
R=\left[ \begin{array}{cc}
           R_1 & 0 \\ 0 & I_2
           \end{array}
\right] , \qquad
R_1^T=\left[ \begin{array}{c}
           R_{1:1} \\ R_{2:1} \\ \vdots \\ R_{n:1}
           \end{array}
\right] .
\end{equation}
\begin{definition}[cf.~\cite{Kraus_12}]
The additive Schur complement approximation (ASCA) of the exact Schur complement
$S=A_{22} - A_{21} A_{11}^{-1} A_{12}$ is denoted by $Q$ and defined as the
Schur complement of
$\widetilde{A}$, i.e.,
$$Q:= \widetilde{A}_{22} -
\widetilde{A}_{21} \widetilde{A}_{11}^{-1} \widetilde{A}_{12}
=\sum_{i=1}^{n} R_{{i}:2}^T (A_{{i}:22} - A_{{i}:21}
A_{{i}:11}^{-1} A_{{i}:12}) R_{{i}:2}.
$$
\end{definition}
\begin{remark}
Note that $\widetilde{A}_{22}=A_{22}$. Thus denoting $\widetilde{N}_1$ and $\widetilde{N}_2$ to be the number 
of fine and coarse DOF on the auxiliary space we have 
$\widetilde{N}_2=N_2$ and $\widetilde{N}_1 \ge N_1$.
\end{remark}

\subsubsection{Auxiliary space two-grid preconditioner}\label{sec:as-2-grid-precond}

The method of fictitious space preconditioning had first been proposed in
\cite{1985MatsokinA_NepomnyashchikhS-aa,Nepomnyaschikh1991mesh,nepomnyaschikh1995fictitious}.
In the following we recall the basic idea.

Let $V=\Reals{N}$ and $\widetilde{V}=\Reals{\widetilde{N}}$ and
define a surjective mapping $\Pi_{\widetilde{D}}: \widetilde{V} \rightarrow V$ by
\begin{equation}\label{eq:Pi}
\Pi_{\widetilde{D}}=(R \widetilde{D} R^T)^{-1} R \widetilde{D},
\end{equation}
where $\widetilde{D}$ is a block-diagonal matrix, e.g.,
\begin{equation}\label{eq:tilde_D}
\widetilde{D}=\left[
             \begin{array}{cc}
              \widetilde{D}_{11} & 0 \\
              0 & I
             \end{array}
           \right],
\end{equation}
e.g., $\widetilde{D}_{11}=\widetilde{A}_{11}$ or
$\widetilde{D}_{11}={\rm diag}(\widetilde{A}_{11})$.

Consider now the fictitious-space two-grid preconditioner $C$ for $A$, which is implicitly defined
in terms of its inverse
\begin{equation}\label{eq:two-grid_0}
C^{-1}=\Pi_{\widetilde{D}} \widetilde{A}^{-1} \Pi^T_{\widetilde{D}}.
\end{equation}

The following spectral equivalence relation follows from application of the fictitious space
lemma, see~\cite{Nepomnyaschikh1991mesh,nepomnyaschikh1995fictitious},
also~\cite{Kraus_Lymb_Mar_2014}.

\begin{lemma} \label{lem:key}
For the preconditioner $C$ defined by~(\ref{eq:two-grid_0})  the following relations hold true
\begin{equation} \label{eq:relBA}
\vek{v}^T C \vek{v} \le \vek{v}^T A \vek{v} \le c_{\Pi} \vek{v}^T C \vek{v} \quad \mbox{for all} \quad \vek{v}\in V,
\end{equation}
\begin{equation}\label{eq:pi}
\kappa(C^{-1}A)\le c_{\Pi}=\Vert \pi_{\widetilde{D}} \Vert^2_{\widetilde{A}},
\quad\mbox{where}\quad
\pi_{\widetilde{D}} := R^T \Pi_{\widetilde{D}} .
\end{equation}
\end{lemma}

\begin{remark}
  A uniform bound on $c_{\Pi}$, independent of the
  contrast or the mesh size, immediately shows that the condition
  number of the preconditioned system is uniformly bounded. 
  However, a theoretical proof of such a robustness result presently
  exists only for particular $H^1$-conforming discretisations of
  second-order scalar elliptic equations with highly heterogeneous
  piecewise constant coefficient,
  see~\cite[e.g., Theorem~4.11]{Kraus_12} and \cite[Theorem 2]{Kraus_Lymb_Mar_2014}. 
  The purely algebraic construction of the preconditioner makes a more 
  general result desirable but also more difficult to prove. However,
  we provide numerical evidence that $c_\Pi$ is bounded
  independently of the contrast (see Table~\ref{table:norm_p_d_tilda}).
\end{remark}

Following the ideas in~\cite{Xu1996auxiliary}, we also consider a more general variant of
the preconditioner~(\ref{eq:two-grid_0}) that incorporates a pre- and post-smoothing process.
Let $M$ denote an $A$-norm convergent smoother, i.e., $\Vert I-M^{-1}A\Vert_A < 1$, and
$\overline{M}=M(M+M^T-A)^{-1}M^T$ the corresponding symmetrised smoother. Examples
of such smoothers include the Gauss-Seidel and damped Jacobi methods and
are well known.

Then an auxiliary space two-grid preconditioner $B$ can be implicitly defined in terms of its inverse
\begin{equation}\label{eq:two-grid_3}
B^{-1} := \overline{M}^{-1} + (I - M^{-T} A) C^{-1} (I - A M^{-1})
\end{equation}
where $C$ is given by~\eqref{eq:two-grid_0}. For a condition number estimate of $B^{-1}A$
we refer to~\cite{Kraus_Lymb_Mar_2014}.

\begin{remark}
The error propagation matrices related to the basic stationary iterative methods
\begin{equation}\label{eq:stat_it}
{\bf{x}}_{k+1}={\bf{x}}_{k} + X {\bf{r}}_k
\end{equation}
with $X=\tau^{-1} C^{-1}$ and $X=\overline{M}^{-1} + \tau^{-1} (I - M^{-T} A) C^{-1} (I - A M^{-1})$,
where ${\bf{x}}_{k}$ and ${\bf{r}}_{k}$ denote the $k$-th iterate and the $k$-th residual,
respectively, are given by
\begin{eqnarray*}
E_C&=&I - \tau^{-1} C^{-1}A \quad \mbox{and} \\ 
E_B&=&I-B^{-1}A=(I-M^{-T}A)(I- \tau^{-1} C^{-1}A)(I-M^{-1}A).
\end{eqnarray*}
Choosing the relaxation parameter $\tau^{-1}$ small enough, i.e.,
$\tau \ge c_{\Pi}=\Vert\pi_{\widetilde{D}} \Vert^2_{\widetilde{A}}$ ensures the
stationary iterative methods~(\ref{eq:stat_it}) to be convergent.
\end{remark}

\subsubsection{Auxiliary space multigrid method}\label{sec:asmg}

Let $k=0,1,\dots,\ell-1$ be the index of mesh refinement where $k=0$ corresponds
to the finest mesh, i.e., $A^{(0)}:=A_h=A$ denotes the fine-grid matrix. Consider
the sequence of auxiliary space matrices $\widetilde{A}^{(k)}$, in the two-by-two
block factorised form
\begin{equation}\label{factorizationK}
({\widetilde{A}}^{(k)})^{-1} =
(\widetilde{L}^{(k)})^T \widetilde{D}^{(k)} \widetilde{L}^{(k)} , 
\end{equation}
where
$$
\widetilde{L}^{(k)} =
\left [
\begin{array}{cc}
I & \\
-\widetilde{A}^{(k)}_{21} (\widetilde{A}^{(k)}_{11})^{-1} & I
\end{array}
\right ] , \quad
\widetilde{D}^{(k)} =
\left [
\begin{array}{cc}
\widetilde{A}^{(k)}_{11} & \\
& Q^{(k)}
\end{array}
\right ]^{-1}
$$
and
the additive Schur complement approximation $Q^{(k)}$ defines the next coarser level matrix, i.e.
\begin{equation}\label{factorizationK2}
 A^{(k+1)}:=Q^{(k)}.
\end{equation}

{
Now define the (nonlinear) AMLI-cycle ASMG preconditioner $C^{(k)}$ at level $k$ by
\begin{equation}\label{multigrid_preconditioner_0}
{C^{(k)}}^{-1} :=
\Pi^{(k)}
 (\widetilde{L}^{(k)})^T 
 \left [
\begin{array}{cc}
\widetilde{A}^{(k)}_{11} & \\
& C_{\nu}^{(k+1)}
\end{array}
\right ]^{-1}
\widetilde{L}^{(k)} {\Pi^{(k)}}^T
\end{equation}
where $\left[C_{\nu}^{(k+1)}\right]^{-1}$ is an approximation of the inverse of the coarse-level
matrix~(\ref{factorizationK2}). At the coarsest level we set
\begin{equation}
\left[C_{\nu}^{(\ell)}\right]^{-1} := {A^{(\ell)}}^{-1}
\end{equation}
and for $ k< \ell-1$ we employ a matrix polynomial of the form
\begin{equation}\label{eq:poly1}
\left[ C_{\nu}^{(k+1)}\right]^{-1} := (I-p^{(k)}({C^{(k+1)}}^{-1}A^{(k+1)})){A^{(k+1)}}^{-1}. 
\end{equation}
If the polynomial $p^{(k)}(t)$ satisfies the condition
$$
p^{(k)}(0)=1
$$
we have the equivalent expression
\begin{equation}\label{eq:poly2}
\left[ C_{\nu}^{(k+1)}\right]^{-1} = q^{(k)}({C^{(k+1)}}^{-1}A^{(k+1)})){C^{(k+1)}}^{-1}
\end{equation}
for (\ref{eq:poly1}) with $q^{(k)}(t)=(1-p^{(k)}(t))/t$ that requires the action of the inverse of $C^{(k+1)}$ only.

A classic choice for $p^{(k)}(t)$ is a scaled and shifted Chebyshev polynomial of degree $\nu_k=\nu$ .
Other polynomials are possible, e.g., choosing $q^{(k)}(t)$ to be the polynomial of best approximation
to $1/t$ in a uniform norm, see~\cite{Kraus2012polynomial}.
}

{
If we incorporate pre- and post-smoothing the AMLI-cycle ASMG preconditioner $B^{(k)}$ at level $k$ is given by
\begin{equation}\label{multigrid_preconditioner}
{B^{(k)}}^{-1} :=
{\overline{M}^{(k)}}^{-1} + (I - {M^{(k)}}^{-T} A^{(k)})
\Pi^{(k)}
 (\widetilde{L}^{(k)})^T  {\overline{D}^{(k)}}^{-1}
\widetilde{L}^{(k)} {\Pi^{(k)}}^T (I - A^{(k)} {M^{(k)}}^{-1})
 \end{equation}
where
$$
\quad \overline{D}^{(k)} :=
\left [
\begin{array}{cc}
\widetilde{A}^{(k)}_{11}& \\
& B_{\nu}^{(k+1)}
\end{array}
\right ] \quad
and \quad
\left[ B_{\nu}^{(k+1)} \right]^{-1} = q^{(k)}({B^{(k+1)}}^{-1}A^{(k+1)})){B^{(k+1)}}^{-1}.
$$

For the nonlinear AMLI-cycle ASMG method
$[B_{\nu}^{(k+1)}]^{-1} \equiv B_{\nu}^{(k+1)}[\cdot] \
(\mbox{or} \ [C_{\nu}^{(k+1)}]^{-1} \equiv C_{\nu}^{(k+1)}[\cdot])$
is a nonlinear mapping whose action on a vector $\vek{d}$ is realised by $\nu$ iterations
using a preconditioned Krylov subspace method.
In the following the generalised conjugate gradient method serves this
purpose and hence we denote $B_{\nu}^{(k+1)}[\cdot] \equiv B^{(k+1)}_{\GCG}[\cdot]$
(and $C_{\nu}^{(k+1)}[\cdot] \equiv C^{(k+1)}_{\GCG}[\cdot]$).
}

\begin{remark}
An important step in the construction of the nonlinear AMLI cycle method is that when performing
$B_{\nu}^{(k+1)}[\cdot]$ one applies~\eqref{multigrid_preconditioner} also for preconditioning
at level $(k+1)$ and hence, \eqref{multigrid_preconditioner} becomes a nonlinear operator, too--we
therefore write $${B^{(k)}}^{-1} \equiv B^{(k)}[\cdot], \quad \mbox{for all } k<\ell.$$
\end{remark}

\subsubsection{Nonlinear ASMG algorithm for the weighted $\bH(\ddiv)$ bilinear form}\label{sec:algorithms}

In the remainder of this section we present the nonlinear ASMG algorithm
for preconditioning the SPD matrices arising from discretisation of the weighted
bilinear form
\eqref{WH-div-prod}
and comment on some details of their implementation when specifically using
lowest-order Raviart-Thomas elements on rectangles.

In Figure \ref{fig:covering} we give an illustration of the {\it covering of ${\Omega}$} by
overlapping subdomains; Here there are $9$ staggered subdomains each of size $1/2$ of the
original domain $\Omega$.

\begin{figure}[ht!]
 \includegraphics[width=0.2\textwidth]{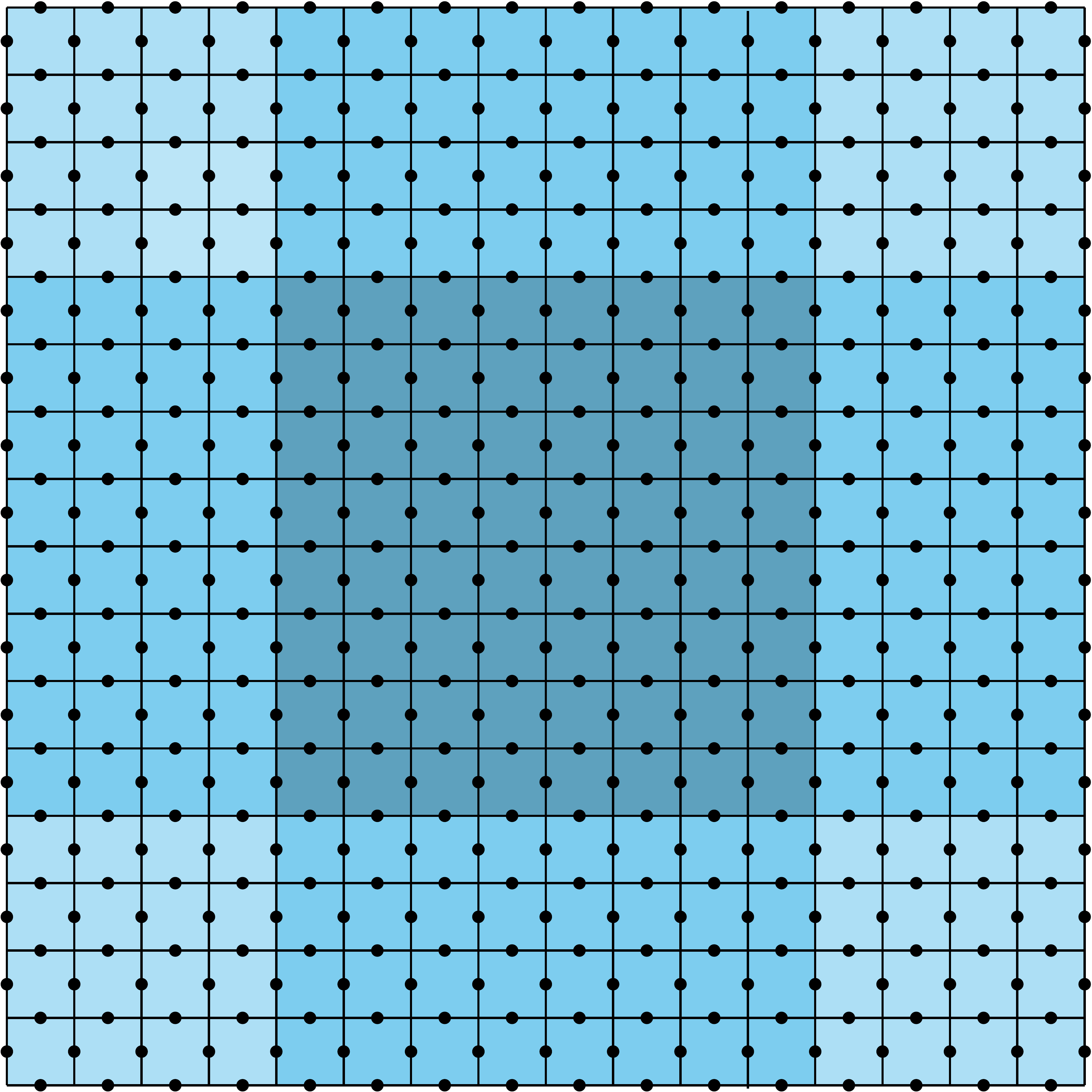}
\caption{Covering of the domain by nine overlapping subdomains}\label{fig:covering}
\end{figure}

Next, consider the {\it partitioning~\eqref{eq:partitioning_DOF} of
the set $\mathcal{D}$} of DOF.  We illustrate for the case of two grids, a coarse
grid ${\mathcal T}_H$,and fine grid ${\mathcal T}_h$, where $H=2h$.
Then the corresponding sets of edges are ${\mathcal E}_H$ and ${\mathcal E}_h$ 
with the following relations being obvious: $4|{\mathcal T}_H|= |{\mathcal T}_h| $
and $2|{\mathcal E}_H|+ 4|{\mathcal T}_H |= |{\mathcal E}_h| $.
Since for lowest-order Raviart-Thomas finite elements it is not immediately
clear how to partition $\mathcal{D}$, we perform a preprocessing step which consists
of a {\it compatible two-level basis transformation} \cite{Kraus_12}.
The global matrix $A$ is transformed
according to
\begin{equation}\label{eq:A_hat}
 \widehat{A}=J^T A J, \quad
\widehat{A}, J,  A  \in \bbR^{ |{\mathcal E}_h| \times  |{\mathcal E}_h| },
\end{equation}
where the transformation matrix $J$ is the product of a permutation matrix $P$ and another
transformation matrix $J_{\pm}$, i.e.
\begin{equation}\label{eq:transformation_J}
 J=P J_{\pm}, \quad P, J_{\pm} \in \bbR^{ |{\mathcal E}_h| \times  |{\mathcal E}_h| }.
\end{equation}

The permutation $P$ allows us to provide
a two-level numbering of the DOF that splits them into two groups,
the first one consisting of DOF associated with fine-grid edges not part of
any coarse-grid  edge
(interior DOF) and the second keeping all remaining DOF ordered such that any two that are
on one and the same coarse edge have consecutive numbers. The transformation matrix $J_{\pm}$
in~\eqref{eq:transformation_J} is of the form
$$
J_{\pm}=\left[\begin{array}{cc}
    I& \\
    & J_{22}
    \end{array}\right], \quad
\mbox{where} \quad I \in \bbR^{(4 |{\mathcal{T}_H}|) \times (4 |{\mathcal{T}_H}|)},
$$
and \quad
$
J_{22}=\frac12 \left[\begin{array}{cccccccc}
    1 & -1 & & & & & & \\
    &  &  1 & -1 & & & & \\
    & & & &  \ddots & \ddots & & \\
    & & & & & & -1 & 1 \\
    1 & 1 & & & & & & \\
    &  &  1 & 1 & & & & \\
    & & & &  \ddots & \ddots & & \\
    & & & & & & 1 & 1
    \end{array}\right],
\quad \mbox{where} \ \ J_{22} \in \bbR^{(2|{\mathcal{E}_H}|) \times (2 |{\mathcal{E}_H}|)}.
$

The analogous transformation is performed at a local level on each subdomain $\Omega_i$, i.e.
\begin{equation}\label{eq:A_i_hat}
\widehat{A}_i=J^T_i A_i J_i,
\end{equation}
where again $J_i=P_i J_{\pm,i}$ with $P_i$ the permutation as explained above but performed
on the degrees of freedom in the subdomain $\Omega_i$. $ J_{\pm,i}$ has its usual
meaning but restricted to the subdomain $\Omega_i$.

\begin{definition}
Global and local transformations are called {\it compatible} if
$$
\widehat{A}=J^T A J = \sum_{i=1}^{n_{\mathcal G}} \widehat{R}_i^T \widehat{A}_i \widehat{R}_i
$$
which is equivalent to the condition
$
R_i J=J_i \widehat{R}_i
$
for all $i$.
\end{definition}

The introduced transformation matrix $J$ defines the splitting of the DOF into coarse and fine, namely
the FDOF correspond to the set of interior DOF and half differences on the coarse edges while the
CDOF correspond to the half sums on the coarse edges.

{This transformation can be applied recursively to coarse-level matrices. The corresponding
two-level transformation matrices, referring to levels $k=0,1,\ldots,\ell-1$, are denoted by $J^{(k)}$.

Finally, the nonlinear ASMG preconditioner employs the following two-level matrices
$$
\widehat{A}^{(k)}={J^{(k)}}^T A^{(k)} J^{(k)}
$$
for all $k<\ell$. Its application to a vector $\widehat{\vek{d}}$ for the two-level basis at level $k$ can be formulated as follows.

\begin{algorithm}
{Action of preconditioner~\eqref{multigrid_preconditioner_0}
on a vector $\widehat{\vek{d}} = {J^{(k)}}^T \vek{d}$ at level $k$: 
$\widehat{C}^{(k)}[\widehat{\vek{d}}]$
}\label{algorithm1} \\[-1ex]
\hrule \vspace{1ex}
\begin{tabular}{lll}
& Auxiliary space correction: & $\left\{
\begin{array}{l}
\left(\begin{array}{c} \tilde{\vek{q}}_1 \\ \tilde{\vek{q}}_2 \end{array}\right)
:=\tilde{\vek{q}} = \Pi_{\widetilde{D}^{(k)}}^T \widehat{\vek{d}} \\
\tilde{\vek{p}}_1 = (\widetilde{A}^{(k)}_{11})^{-1} \tilde{\vek{q}}_1 \\
\tilde{\vek{p}}_2 ={J^{(k+1)}}
C_{\GCG}^{(k+1)}[{J^{(k+1)}}^T (\tilde{\vek{q}}_2 -  \widetilde{A}^{(k)}_{21} \tilde{\vek{p}}_1)] \\
\tilde{\vek{q}}_1 = \tilde{\vek{p}}_1 - (\widetilde{A}^{(k)}_{11})^{-1} \widetilde{A}^{(k)}_{12} \tilde{\vek{p}}_2 \\
\tilde{\vek{q}}_2 = \tilde{\vek{p}}_2 \\
\widehat{C}^{(k)}[\widehat{\vek{d}}] := \Pi_{\widetilde{D}^{(k)}} \tilde{\vek{q}}
\end{array} \right.$ \\
\end{tabular}
\\[1ex]
\hrule
\end{algorithm}

By incorporating pre- and post-smoothing the realisation of the preconditioner~\eqref{multigrid_preconditioner}
takes the following form.

\begin{algorithm}
{Action of preconditioner~\eqref{multigrid_preconditioner}  on a vector 
$\widehat{\vek{d}}$ at level $k$: $\widehat{B}^{(k)}[\widehat{\vek{d}}]$}\label{algorithm2} \\[-1ex]
\hrule \vspace{1ex}
\begin{tabular}{lll}
& Pre-smoothing: & $\widehat{\vek{u}} = (\widehat{M}^{(k)})^{-1} \widehat{\vek{d}}$ \\
& Auxiliary space correction: & 
$\widehat{\vek{v}} = \widehat{\vek{u}} + \widehat{C}^{(k)}[\widehat{\vek{d}} - \widehat{A}^{(k)} \widehat{\vek{u}}]$ \\
& Post-smoothing: & $\widehat{B}^{(k)}[\widehat{\vek{d}}] :=
\widehat{\vek{v}} + (\widehat{M}^{(k)})^{-T} (\widehat{\vek{d}} - \widehat{A}^{(k)}
\widehat{\vek{v}})$
\end{tabular}
\\[1ex]
\hrule
\end{algorithm}

}

\subsection{On the complexity of the ASMG preconditioner}\label{sec:complexity}

  We now address the important topic of estimating the computational
  work required for performing the action of the ASMG preconditioner.
  Clearly, from the algorithm descriptions given earlier, the number
  of flops required to evaluate such an action is proportional to the
  \emph{operator complexity} of the preconditioner, defined
  as the total number of non-zeroes in the matrices on all levels.  The
  most general algebraic multilevel preconditioners are usually
  constructed using the combinatorial graph structure of the
  underlying matrices (on the finest and coarser levels).  Estimating
  the operator complexities in such cases is not only difficult, but
  in most cases impossible due to the fact that such estimates should
  hold for the set of \emph{all} possible graphs. Reliable estimates
  are usually done for algorithms that construct coarse levels using
  at least some of the geometric information from the underlying
  problem. This is the case we consider here, and we also refer
  to~\cite{2012BrezinaM_VanekP_VassilevskiP-aa}, 
  \cite{2013WangL_HuX_CohenJ_XuJ-aa} for more insight into how  
  geometric information can be used to bound the operator complexity
  of a multilevel preconditioner.

  Given a matrix $A\in \mathbb{R}^{N\times N}$, we
  characterize the nonzero structure of the ASMG coarse level matrix
  $Q$.
To construct $Q$, recall that we first need to split the set of the
degrees of freedom as a union of subsets,
$\{1,\ldots,N\} = \cup_{i=1}^n\omega_i$.  We assume that 
$\omega_i=\{\mathcal{F}_i,\mathcal{C}_i\}$, where 
$\mathcal{F}_i$ is a set of fine grid degrees of freedom, 
and, $\mathcal{C}_i$, is a set of coarse grid degrees of freedom, with 
$\mathcal{F}_i\cap\mathcal{C}_i = \emptyset$. The total number of coarse grid degrees of freedom is $N_{\mathcal{C}}=\left|\cup_{j=1}^n \mathcal{C}_i \right|$. Since our
considerations are permutation invariant, without loss of generality, we assume that globally we have numbered first the
coarse grid degrees of freedom, and thus we have
$\mathcal{C}_i\subset\{1,\ldots,N_{\mathcal{C}}\}$.

We also set $N_i=|\omega_i|$, and $n_i=|\mathcal{C}_i|$.  We denote by
${\bf e}_k$ the $k$-th Euclidean basis vector in $\mathbb{R}^N$; when
we consider the canonical basis in $\mathbb{R}^m$, $m\neq N$, we use
the notation ${\bf e}_{k(m)}$ for the $k$-th basis vector. 
With each $\omega_i$  we  associate a matrix 
$R_i\in\mathbb{R}^{N_i\times N}$, and, for $\omega_i=\{j_1,\ldots,j_{N_i}\}$, we set
$R^T_i = [{\bf e}_{j_1},\ldots,{\bf e}_{j_{N_i}}]$.  
Next, we consider a fine grid matrix $A$ given by the identity
\[
A = \sum_{i=1}^n R^T_i A_i R_i= \sum_{i=1}^n
\begin{bmatrix}
R^T_{i,\mathcal{F}},R^T_{i,\mathcal{C}}
\end{bmatrix} 
\begin{bmatrix}
A_{i,\mathcal{F}} & A_{i,\mathcal{FC}} \\ 
A_{i,\mathcal{CF}} & A_{i,\mathcal{C}}
\end{bmatrix} 
\begin{bmatrix}
R_{i,\mathcal{F}}\\R_{i,\mathcal{C}}
\end{bmatrix} 
\]
where we used a block form of the matrices corresponding to 
the splitting of $\omega_i$ on $\mathcal{F}$-ine level and
$\mathcal{C}$-oarse level degrees of freedom. The Schur complements
$S_i$ used in the definition of the coarse grid matrix $Q$ 
are defined as $S_i =
A_{i,\mathcal{C}}-A_{i,\mathcal{CF}}A_{i,\mathcal{F}}^{-1}A_{i,\mathcal{FC}}$. 
Recall that the coarse grid matrix $Q$ then is
defined by
\(
Q = \sum_{i=1}^n \widetilde{R}^T_{i,\mathcal{C}}S_i \widetilde{R}_{i,\mathcal{C}}, \quad\mbox{and we have}\quad
Q\in\mathbb{R}^{N_{\mathcal{C}}\times N_{\mathcal{C}}}.
\) 
If we now use our assumption that the coarse grid degrees of freedom are numbered first,  
then $\widetilde{R}_{i,\mathcal{C}}\in \mathbb{R}^{n_i\times N_\mathcal{C}}$
is formed by the first $N_{\mathcal{C}}$ columns of
$R_{i,\mathcal{C}}\in \mathbb{R}^{n_i\times N_{\mathcal{C}}}$.  Next, we introduce
the vectors \( \bm{1}_i = (\underbrace{1,\ldots,1}_{n_i})^T\), and
\(\bm{\chi}_i = \sum_{j\in \mathcal{C}_i}
{\bf e}_{j(N_{\mathcal{C}})}\).
For a fixed $i$, the vector $\bm \chi_i\in \mathbb{R}^{N}$ is the
indicator vector of the set $\mathcal{C}_i$ as a subset of
$\{1,\ldots,N_{\mathcal{C}}\}$. Its components are equal to $1$ for
indicies in $\mathcal{C}_i$ and equal to zero otherwise. We note that
$\bm1_i\bm1_i^T$ is the $n_i\times n_i$ matrix of all ones, and we
encourage the reader to check the identity
$\bm{\chi}_i = \widetilde R_{i,\mathcal{C}}^T \bm{1}_i$.

To describe the nonzero structure of $Q$ we introduce
the set $\mathcal{B}_m$ of Boolean $(m\times m)$ matrices whose
entries are from the set $\{0,1\}$. We introduce a mapping
$\operatorname{nz}:\mathbb{R}^{m\times m}\mapsto \mathcal{B}_{m}$,
such that $[\operatorname{nz}(A)]_{ij} = 0$ if and only if $A_{ij}=0$
and $[\operatorname{nz}(A)]_{ij} = 1$ otherwise.  We say that
$X\preceq Y$ if $[\operatorname{nz}(Y)-\operatorname{nz}(X)]$ is a
matrix with non-negative entries.  This is a formal way to state that the nonzero
structure of $Y$ contains the nonzero structure of $X$, or, to say
that every zero in $Y$ is also a zero in $X$. 
Clearly, $S_i\preceq \bm1_i\bm1_i^T$, and, 
as a consequence, we have the following relation characterizing the  sparsity of $Q$:
\begin{equation}\label{nzQ}
Q\preceq \sum_{i=1}^n 
\sum_{i=1}^n \widetilde R_{i,\mathcal{C}}^T \bm{1}_i\bm{1}_i^T
\widetilde R_{i,\mathcal{C}} = 
\sum_{i=1}^n\bm{\chi}_i\bm{\chi}_i^T=:X.
\end{equation}
Note that from the right side of~\eqref{nzQ} we can conclude that
$Q_{km}$ may be nonzero only in the case when there exists $i$ such that
$k\in \mathcal{C}_i$ and $m\in \mathcal{C}_i$.  
Using~\eqref{nzQ} it is easy to compute a bound on
the number of non-zeroes $n_{z,j}$, for fixed column $j$ in $Q$. We have
\[
n_{z,j}\le \|X{\bf e}_{j(N_{\mathcal{C}})}\|_{\ell^1} = \sum_{i:j\in\mathcal{C}_i}|\mathcal{C}_i|.
\]
As is immediately seen, the number of non-zeroes per column in $Q$ is
bounded by a constant independent of $N$ if the following two
conditions are satisfied: (i) the number of coarse grid degrees of
freedom in each $\mathcal{C}_i$ is bounded; and (ii) every coarse grid
degree of freedom lies in a bounded number of subsets
$\mathcal{C}_i$. 

As a simple, but instructive example how the conditions (i) and (ii) can be
satisfied, let us consider a PDE discretized by FE method on a
quasiuniform grid with characteristic mesh size $h$ in 2D. The
considerations are independent of the PDE or the order of the FE
spaces (but the constants hidden in ``$\lesssim$'' below may depend on
the FE spaces and the order of polynomials). 
To define the sets $\omega_i$ on such a grid, we proceed as follows: 
(1) place a regular (square) auxiliary grid of size $\gamma h$, $\gamma\ge 2$ that
contains $\Omega$; (2) set $n$ to be the number of vertices on the
auxiliary grid, lying in $\Omega$; (3)~choose $\omega_i$ to be the
set of DOF
degrees of freedom 
lying in the support of the bilinear basis
function corresponding to the $i$-th vertex. Then we have that
$|\mathcal{C}_i| < N_i\lesssim 4\gamma$ and every coarse grid degree of
freedom lies in at most $4$ such subdomains. The constant hidden in
``$\lesssim$'' is a bound on the number of degrees of freedom lying in
a square of size $2h$. The fact that this bound depends only on 
the polynomial order and type of FE spaces follows from the assumption 
that the mesh is quasi-uniform. For efficient
and more sophisticated techniques using regular, but adaptively
refined, auxiliary grids in coarsening algorithms for unstructured 
problems we refer to~\cite{2012BrezinaM_VanekP_VassilevskiP-aa},
\cite{2013WangL_HuX_CohenJ_XuJ-aa}. Such techniques may directly
be applied to yield optimal operator complexities for the ASMG
preconditioner in the general case of shape regular grids, 
albeit the details are beyond the scope of our consideration here. 

\section{Numerical Experiments}\label{sec:numerics}

\subsection{Description of the parameters and the numerical test examples}

Subject to numerical testing are three representative cases characterised by a highly
varying coefficient $\al(x) = K^{-1}(x)$, namely:
\begin{enumerate}
\item[[a\hspace{-1ex}]] A binary distribution of the coefficient described by islands on which
 $\al=1.0$ against a background where $\al=10^{-q}$, see~Figure~\ref{fig:islands_binary};
 \item[[b\hspace{-1ex}]] Inclusions with $\al=1.0$ and a background with a coefficient
 $\al=\al_{T}=10^{-q_{rand}}$ that is constant on each element $\tau \in  {\mathcal T}_h$,
 where the random integer exponent $q_{rand}\in\{0,1,2,\dots,q\}$ is uniformly distributed,
see Figure~\ref{fig:islands_random};
\item[[c\hspace{-1ex}]] Three two-dimensional slices of the SPE10
benchmark problem, 
where the contrast $\kappa$ is $10^7$ for slices 44 and 74 and $10^6$ for slice~54,
see~Figure~\ref{fig:spe_10}.
\end{enumerate}
Test problems [a] and [b] are similar to those considered in other works, e.g.~\cite{efendiev2012robust,
Kraus_Lymb_Mar_2014,
Willems_SINUM_2014}. Example~[c] consists of 2-D slices 
of 3-dimensional data of SPE10 (Society of Petroleum Engineers)
benchmark, see~\cite{SPE10_project}.

\begin{figure}[hb]
\begin{center}
\subfigure[$32\times 32$ mesh]{
\includegraphics[width=0.21\textwidth]{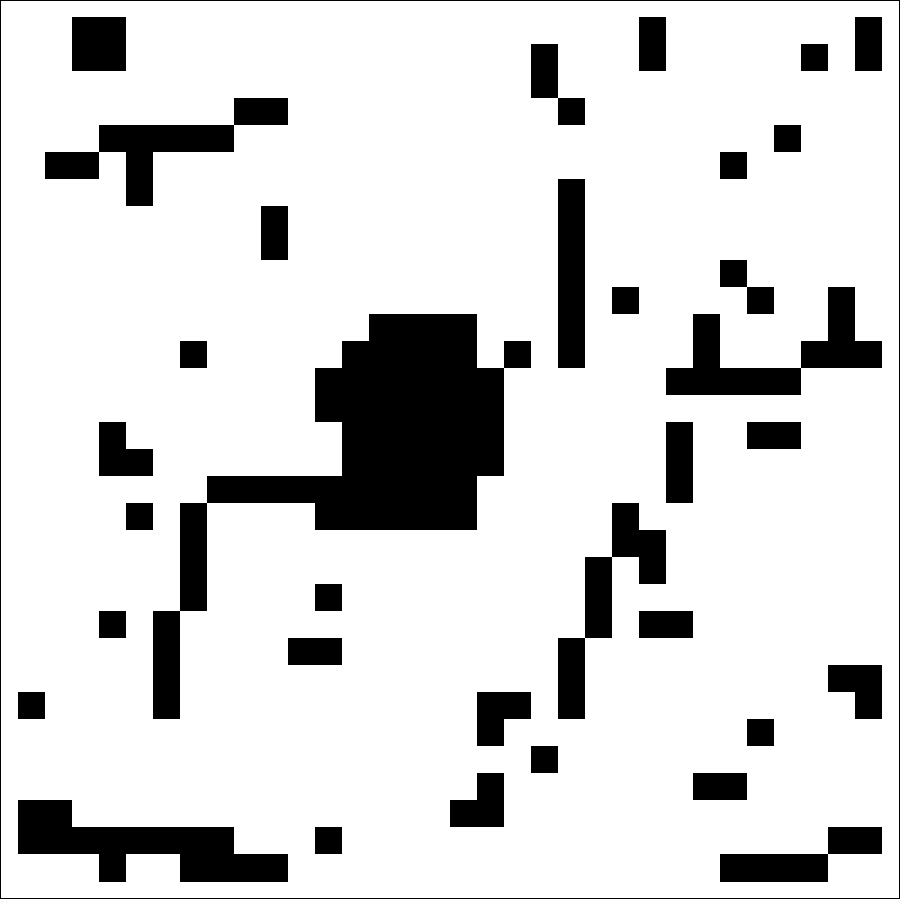}
\label{fig:islands_32}}
\hspace{10mm}
\subfigure[$128\times 128$ mesh]{
\includegraphics[width=0.21\textwidth]{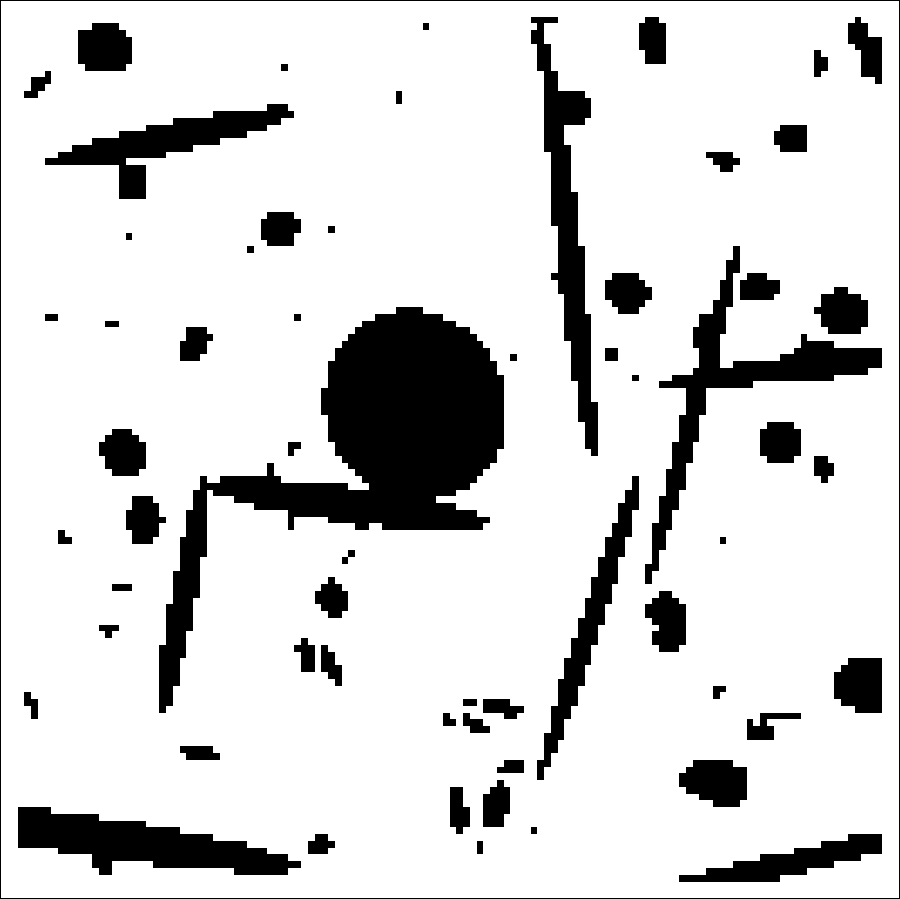}
\label{fig:islands_128}}
\hspace{10mm}
\subfigure[$512\times 512$ mesh]{
\includegraphics[width=0.21\textwidth]{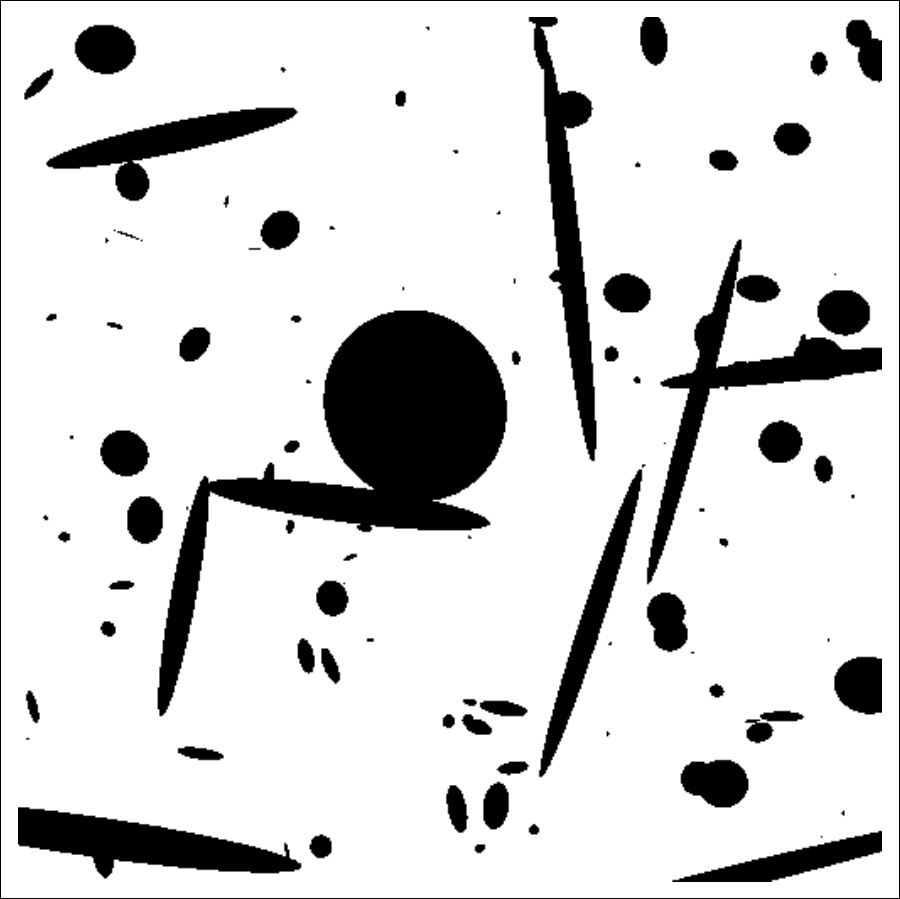}
\label{fig:islands_512}}
\caption{Binary distribution of the permeability $K(x)$ corresponding to test case [a]}
\label{fig:islands_binary}
 \end{center}
\end{figure}

\begin{figure}[hb]
\begin{center}
\subfigure[$32\times 32$ mesh]{
\includegraphics[width=0.21\textwidth]{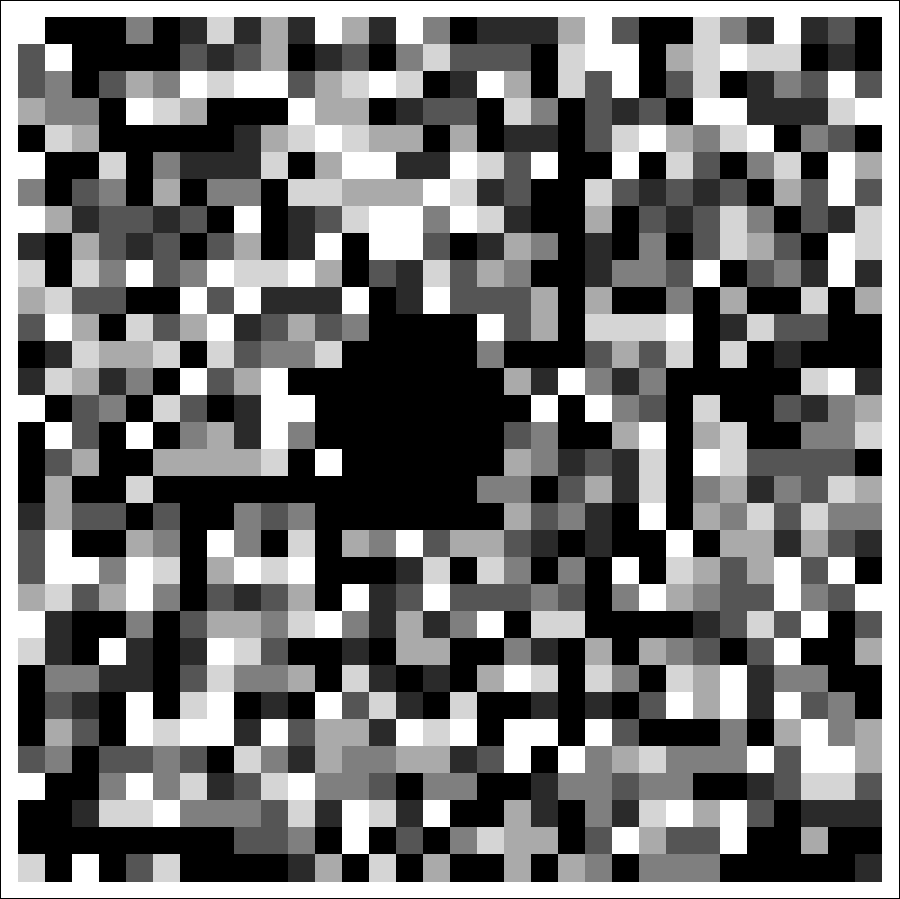}
\label{fig:islands_32r}}
\hspace{10mm}
\subfigure[$128\times 128$ mesh]{
\includegraphics[width=0.21\textwidth]{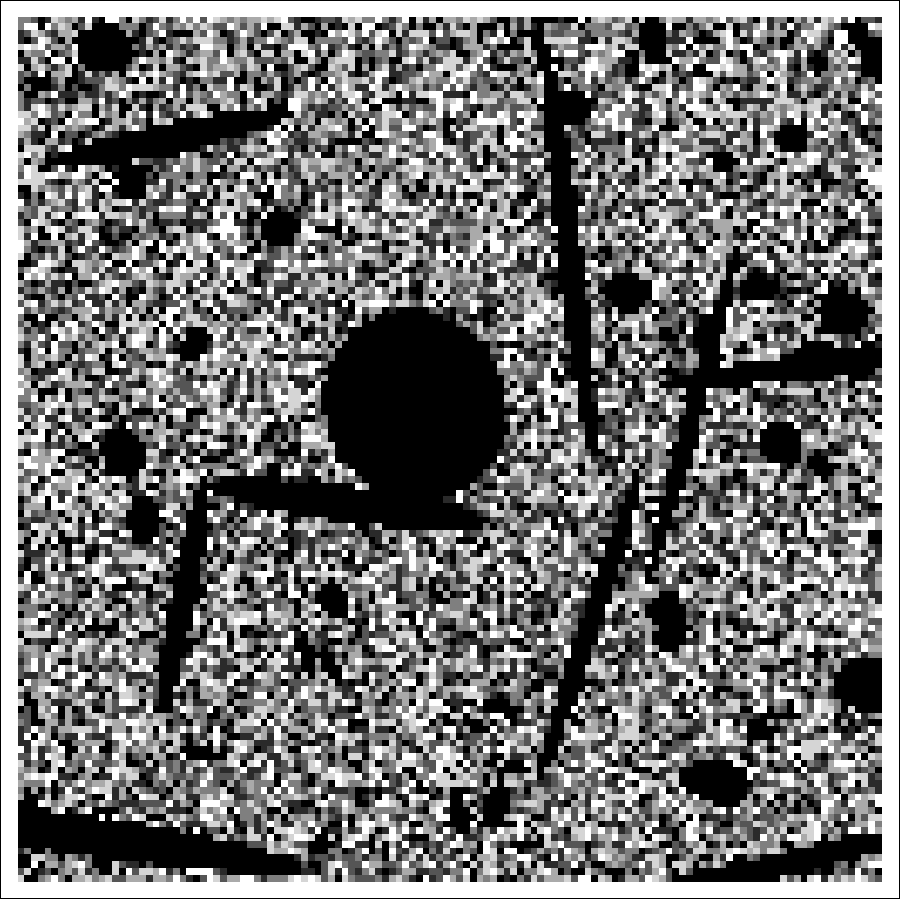}
\label{fig:islands_128r}}
\hspace{10mm}
\subfigure[$512\times 512$ mesh]{
\includegraphics[width=0.21\textwidth]{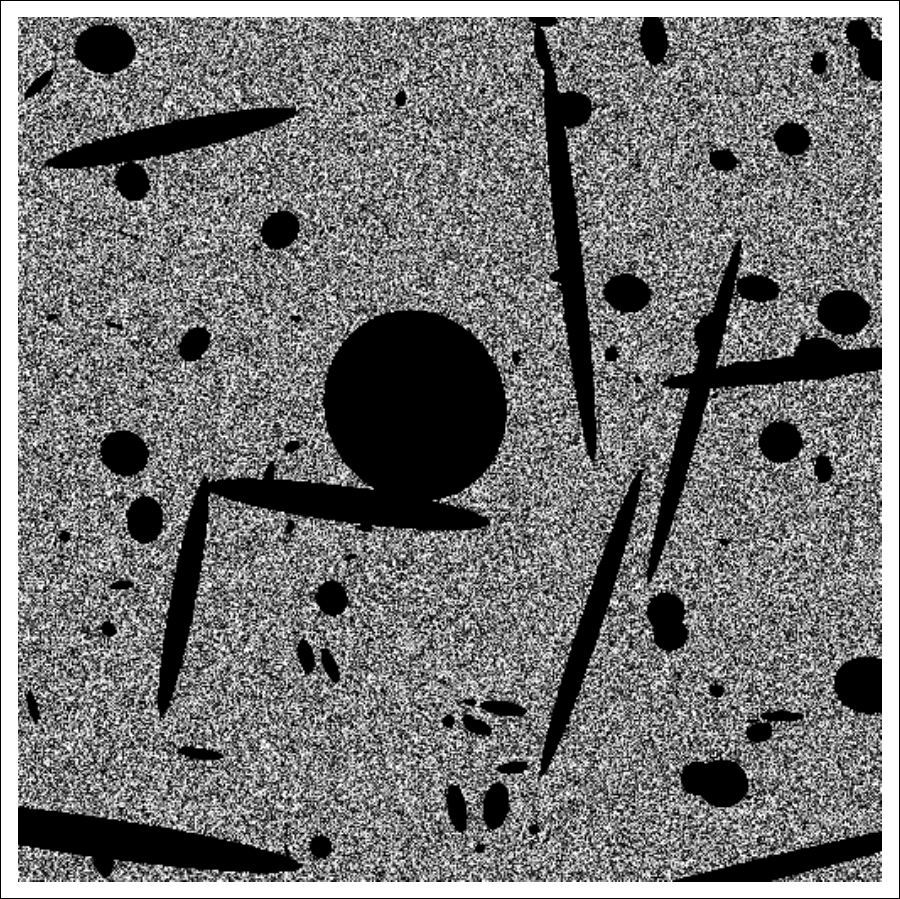}
\label{fig:islands_512r}}
\caption{Random distribution of  $\al=K^{-1}(x)$ corresponding to test case [b]}
\label{fig:islands_random}
 \end{center}
\end{figure}

\begin{figure}[hb]
\begin{center}
\subfigure[Slice 44]{
\includegraphics[width=0.28\textwidth]{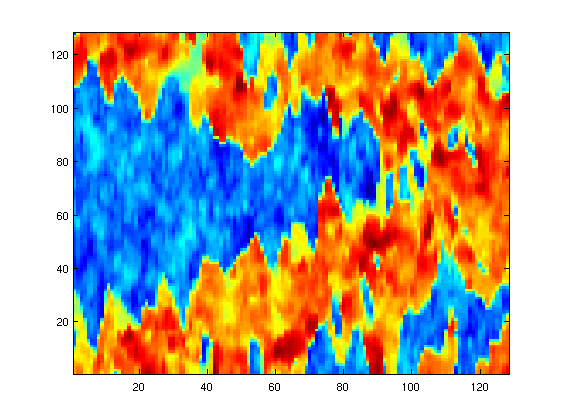}
\label{fig:slice44}}
\hspace{3mm}
\subfigure[Slice 54]{
\includegraphics[width=0.28\textwidth]{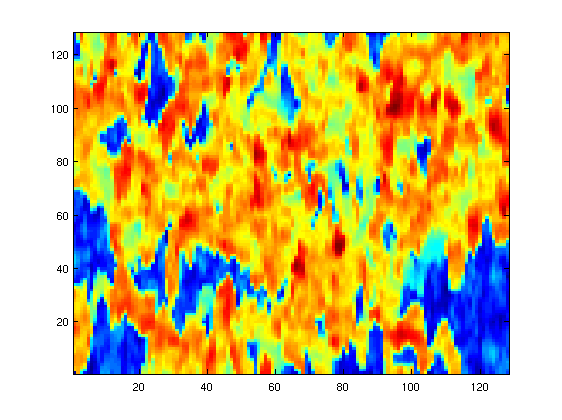}
\label{fig:slice54}}
\hspace{3mm}
\subfigure[Slice 74]{
\includegraphics[width=0.28\textwidth]{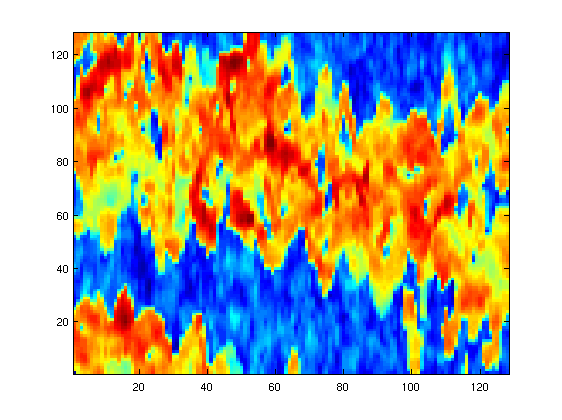}
\label{fig:slice74}}
\caption{Distributions of the permeability $K(x)$ along planes $x_3=44, 54, 74$ from
the benchmark SPE10 on a $ 128 \times 128 $ mesh}
\label{fig:spe_10}
 \end{center}
\end{figure}

The numerical experiments were performed over a uniform mesh consisting of $N{\times}N$ square elements 
where $N=4,8,\ldots,512$, i.e. up to $525 312$ velocity DOF and $262 144$ pressure DOF.
We have used a direct method to solve the problems on the coarsest grid.
The iterative process has been initialised with a random vector.
Its convergence has been tested for linear systems with the right hand side zero
{except in the last example where we have solved the mixed system
\eqref{equation-1} for slice 44 of the SPE10 problem with the right hand side
\begin{equation}\label{eq:rhs}
f=\left\{
\begin{array}{rcl}
c & \mbox{for} & (x,y) \in \Omega_{+}=[0.2,0.3]{\times}[0.7,0.8] \\
-c & \mbox{for} & (x,y) \in \Omega_{-}=[0.7,0.8]{\times}[0.2,0.3] \\
0 & \mbox{for} & (x,y) \in \Omega \setminus (\Omega_{+} \cup \Omega_{-})
\end{array}
\right.
\end{equation}
}

We have used overlapping coverings of the domain as shown in Figure~\ref{fig:covering},
where the subdomains are composed of $8\times 8$ elements and overlap with half their
width or height.
When presenting results we use the following notation:
\begin{itemize}
\item $\ell$ denotes the number of levels;
\item $q=\frac{\log \kappa}{\log 10}$ is the logarithm of the contrast $\kappa$;
\item $n_{ASMG}$ is the number of auxiliary space multigrid iterations;
\item $m\geq 0$ is the number of
point Gauss-Seidel pre- and post-smoothing steps;
{
\item $\rho_r$ is the average residual reduction factor defined by
\begin{equation}\label{eq:average_factor}
\rho_r=\Bigg(\frac{\Vert \vek{r}_{n_{ASMG}} \Vert}{\Vert \vek{r}_0 \Vert}\Bigg)^{1/n_{ASMG}},
\end{equation}
where
$\vek{u}_{n_{ASMG}}$ is the
first iterate (approximate solution of \eqref{algebraic-hdiv}) for which the initial residual
has decreased by a factor of at least $10^8$;
\item $\rho_e := \Vert I - {{C}^{(0)}}^{-1} A^{(0)} \Vert_{A^{(0)}}$ is the norm of the error propagation matrix of the
linear V-cycle preconditioner~\eqref{multigrid_preconditioner_0} which is obtained by choosing
the polynomial $p_\nu(t)=1-t$ in \eqref{eq:poly1}.
}
\end{itemize}
The matrix $\widetilde{D}$ is as in~\eqref{eq:tilde_D} where
$\widetilde{D}_{11}=\widetilde{A}_{11}$. This choice of $\widetilde{D}$ requires an
additional preconditioner for the iterative solution of linear systems with the matrix
$D=R \widetilde{D} R^T$ a part of the efficient application of the operator
$\Pi_{\widetilde{D}}$. The systems with $D$ are solved using the preconditioned conjugate
gradient (PCG) method. The stopping criterion for this inner iterative process is a
residual reduction by a factor $10^6$, the number of PCG iterations to reach it--where
reported--is denoted by $n_i$. The preconditioner $B_{ILUE}$ for $D$ is constructed
using incomplete factorisation with exact local factorisation (ILUE).
The definition of $B_{ILUE}$ is as follows:
$$
B_{ILUE}:=LU,\qquad U:=\sum_{i=1}^n R_i^T U_i R_i,\qquad L:=U^T {\rm diag}(U)^{-1},
$$
where
$$
D_i=L_i U_i,\qquad D=\sum_{i=1}^n R_i^T D_i R_i, \qquad {\rm diag}(L_i)=I.
$$
For details see~\cite{Kraus2009robust}.
Note that as $D_i$ are the local contributions to $D$ related to the subdomains
$\Omega_i$, $i=1,\dots,n$, they are all non-singular.

The following two sections are devoted to the presentation of numerical results.
Experiments fall into two categories.
{
The first category, presented in Section~\ref{sec:Hdiv}, serves the evaluation of the
performance of the ASMG method on linear systems arising from discretisation of the
weighted $\bH(\ddiv)$ bilinear form~\eqref{WH-div-prod}. All three test cases, [a], [b],
and [c], are considered testing V- and W-cycle methods with and without smoothing.
Additionally, we evaluate as a robustness indicator the quantity
$\Vert \pi_{\widetilde{D}} \Vert^2_{\widetilde{A}}=\Vert R^T \Pi_{\widetilde{D}} \Vert^2_{\widetilde{A}}$,
which appears in \eqref{eq:relBA} and \eqref{eq:pi}.
}

The second category of experiments, discussed in Section~\ref{sec:saddle_point_sys},
addresses the solution of the indefinite linear system~\eqref{eq:saddle_point_sys-intro}
arising from problem~\eqref{eq:dual_mixed} by a preconditioned MinRes method.
The main aims are, on the one hand, to confirm the robustness of the block-diagonal
preconditioner~\eqref{eq:AFW_preconditioner} with respect to arbitrary multiscale
coefficient variations, and on the other, to demonstrate its numerical scalability.
Furthermore, we include a test problem with a nonzero right hand side.

\subsection{Numerical tests for solving the system~\eqref{algebraic-hdiv}}\label{sec:Hdiv}
An ASMG preconditioner was tested for solving the system~\eqref{algebraic-hdiv} with
a matrix corresponding to discretisation of the form $\Lambda_\alpha(\bu,\bv)$.

{
\begin{example} \label{ex:0}
First we estimate $\Vert\pi_{\widetilde{D}}\Vert^2_{\widetilde{A}}$ at
the level of the finest mesh, level $0$, for increasing contrast of
magnitude $10^q$ in the configuration of case [b].
\end{example}

The experimental study is based on Lemma \ref{lem:key}.
The quantity of interest in Example~\ref{ex:0} provides an upper bound
for the condition number $\kappa(C^{-1}A)$ for the two-level
preconditioner~\eqref{eq:two-grid_0}. The results shown in
Table~\ref{table:norm_p_d_tilda} clearly demonstrate 
that $\kappa(C^{-1}A)$ is uniformly bounded.

This test indicates the robustness of the fictitious space
preconditioner with respect to a highly varying coefficient on multiple
length scales. 
Note that test case [b] is considered to be
representative, taking into account the iteration counts for the test
cases [a] and [b] as presented in the following examples.
}

\begin{table}[ht!]
\begin{center}
\begin{tabular}{c|  c |  c | c  | c |  c }
 \multicolumn {6}{c}{Value of $\Vert\pi_{\widetilde{D}}\Vert^2_{\widetilde{A}}$: bilinear form~\eqref{WH-div-prod}} \\ \multicolumn{6}{c}{~} \\
 & $\ell=3$ & $\ell=4$ & $\ell=5$ & $\ell=6$ & $\ell=7$ \\
\cline{2-6}
\hline 
$q = 0$   & $1.122$  & $1.137$  & $1.148$  & $1.150$  & $1.149$ \\
$q = 1$   & $1.148$  & $1.169$  & $1.149$  & $1.152$  & $1.138$ \\
$q = 2$   & $1.286$  & $1.338$  & $1.360$  & $1.287$  & $1.126$ \\
$q = 3$   & $1.336$  & $1.389$  & $1.418$  & $1.326$  & $1.132$ \\
$q = 4$   & $1.343$  & $1.396$  & $1.426$  & $1.334$  & $1.133$ \\
$q = 5$   & $1.343$  & $1.397$  & $1.426$  & $1.333$  & $1.369$ \\
$q = 6$   & $1.343$  & $1.397$  & $1.426$  & $1.333$  & $1.369$ \\
\end{tabular} \vspace{2ex}
\caption{Example~\ref{ex:0}: case [a] with $K(x)=10^q$}\label{table:norm_p_d_tilda_}
\end{center}
\end{table}

\vspace{-1ex}

\begin{table}[ht!]
\begin{center}
\begin{tabular}{c|  c |  c | c  | c |  c }
 \multicolumn {6}{c}{Value of $\Vert\pi_{\widetilde{D}}\Vert^2_{\widetilde{A}}$: bilinear form~\eqref{WH-div-prod}} \\ \multicolumn{6}{c}{~} \\
 & $\ell=3$ & $\ell=4$ & $\ell=5$ & $\ell=6$ & $\ell=7$ \\
\cline{2-6}
\hline 
$q = 0$   & $1.122$  & $1.137$  & $1.148$  & $1.150$  & $1.149$ \\
$q = 1$   & $1.115$  & $1.126$  & $1.169$  & $1.167$  & $1.123$ \\
$q = 2$   & $1.126$  & $1.208$  & $1.119$  & $1.146$  & $1.112$ \\
$q = 3$   & $1.014$  & $1.261$  & $1.338$  & $1.334$  & $1.110$ \\
$q = 4$   & $1.260$  & $1.295$  & $1.371$  & $1.434$  & $1.110$ \\
$q = 5$   & $1.268$  & $1.329$  & $1.394$  & $1.493$  & $1.145$ \\
$q = 6$   & $1.255$  & $1.374$  & $1.412$  & $1.139$  & $1.113$ \\
\end{tabular} \vspace{2ex}
\caption{Example~\ref{ex:0}: case [b] with $K(x)=10^q$}\label{table:norm_p_d_tilda}
\end{center}
\end{table}

{
\begin{example} \label{ex:1}
Next we are interested in the convergence factor in $A$-norm of the linear V-cycle method, that is,
we evaluate the quantity $\rho_e := \Vert I - {{C}^{(0)}}^{-1} A^{(0)} \Vert_{A^{(0)}}$. Moreover, we 
compare $\rho_e$ to the corresponding value of the average residual reduction factor $\rho_r$ defined
according to \eqref{eq:average_factor}. We also report the number of iterations $it_e$ that reduce the
initial error in $A$-norm by a factor $10^{8}$ and the number of iterations $it_r$ that reduce the Euclidean
norm of the initial residual by the same factor. The problem configuration is test case [c] for a zero right
hand side.
\end{example}

The results for Example~\ref{ex:1} are summarised in Table~\ref{table:linear_V_cycle}. Although the
average residual reduction factor $\rho_r$ is much smaller than $\rho_e$, the error reduction in the 
$A$-norm is also surprisingly good, especially in view of the linear V-cycle performing 
without any additional smoothing, i.e., implementing Algorithm~\ref{algorithm1}.

}


\begin{table}[ht!]
\begin{center}
\begin{tabular}{c|  c | c | c  | c  }
\multicolumn {5}{c}{Linear V-cycle: bilinear form~\eqref{WH-div-prod}} \\ \multicolumn{5}{c}{~} \\
 & $\rho_e$ & $it_e $ & $\rho_r $ & $it_r $ \\
\cline{2-5}
\hline 
$\ell = 4$   & $0.105$  & $7$  & $0.031$  & $6$   \\
$\ell = 5$   & $0.289$  & $9$  & $0.095$  & $8$   \\
$\ell = 6$   &  $0.494$  & $12$  & $0.168$ & $11$   \\
$\ell = 7$   & $~~0.642~~$  & $~~~14~~~$  & $~~0.215~~$  & $~~~12~~~$   \\
$\ell = 8$   & $0.729$  & $17$  & $0.262$  & $14$   \\
\end{tabular} \vspace{2ex}
\caption{Example~\ref{ex:1}: case [c] - slice 44 of the SPE10 benchmark.}\label{table:linear_V_cycle}
 \end{center}
\end{table}


{
\begin{example}\label{ex:2}
Now we test the nonlinear V-cycle and the effect of additional smoothing. 
Again the problem configuration is test case [c] for a zero right hand side. We report the
number of nonlinear AMLI-cycle ASMG iterations with Algorithm~\ref{algorithm1} denoted
by $n_{ASMG}$ for a residual reduction by eight orders of magnitude along with $\rho_r$.
\end{example}

Comparing the results for Example~\ref{ex:2}, which are listed in Table~\ref{table:non_linear_V_cycle},
with those in Table~\ref{table:linear_V_cycle} shows that the nonlinear V-cycle typically also reduces the
residual norm faster than the linear V-cycle--for the reduction of the $A$-norm of the error this is a known
fact--and the additional incorporation of a point Gauss-Seidel relaxation further accelerates the convergence. 

}


\begin{table}[ht!]
\begin{center}
\begin{tabular}{c| c  c | c  c | c c   }
 \multicolumn {7}{c}{Non-linear V-cycle: bilinear form~\eqref{WH-div-prod}} \\
\multicolumn{1}{c}{~} & \multicolumn{2}{ c |}{$m=0$} & \multicolumn{2}{c|}{$m=1$}
& \multicolumn{2}{c}{$m=2$} 
\\
\cline{2-7}
& $n_{ASMG}$ & $\rho$ & $n_{ASMG}$ & $\rho$ & $n_{ASMG}$ & $\rho$   \\
\hline 
$\ell = 4$   & $6$ & $0.032$  & $5$ & $0.025$  & $6$  & $0.027$   \\
$\ell = 5$   & $8$ & $0.093$  & $7$ & $0.062$  & $6$  & $0.045$   \\
$\ell = 6$   & $11$ & $0.157$  & $8$ & $0.091$  & $8$ & $0.083$   \\
$\ell = 7$   & $12$ & $0.202$  & $9$ & $0.123$  & $8$ & $0.094$   \\
$\ell = 8$   & $14$ & $0.243$  & $11$ & $0.172$  & $10$ & $0.154$   \\
\end{tabular} \vspace{2ex}
\caption{Example~\ref{ex:2}: case [c] - slice 44 of the SPE10 benchmark.}\label{table:non_linear_V_cycle}
 \end{center}
\end{table}


\begin{example}\label{ex:3}
The next example tests the dependency of the convergence rate with respect to the contrast.
The configuration is test case [a] containing a zero right hand side and number of smoothing steps 
$m=2$.
\end{example}
Results in Table~\ref{table:a_bilinear_alg1_V_m2} show a slight increase in $\rho_r$ with increasing
contrast.
\begin{table}[h!]
 \begin{center}
 \begin{tabular}{c| c  c | c  c | c c  | c c | c  c }
 \multicolumn {11}{c}{ASMG V-cycle: bilinear form~\eqref{WH-div-prod},
Algorithm~\ref{algorithm1}} \\
\multicolumn{1}{c}{~} & \multicolumn{2}{ c |}{$\ell=3$} & \multicolumn{2}{c|}{$\ell=4$}
& \multicolumn{2}{c|}{$\ell=5$} & \multicolumn{2}{c|}{$\ell=6$}
& \multicolumn{2}{c}{$\ell=7$} 
\\
\cline{2-11}
& $n_{ASMG}$ & $\rho$ & $n_{ASMG}$ & $\rho$ & $n_{ASMG}$ & $\rho$   &   $n_{ASMG}$ & $\rho$ & $n_{ASMG}$ & $\rho$ \\ 
\hline 
$q = 0$   & $4$ & $0.005$  & $5$ & $0.024$  & $6$ & $0.043$  & $8$ & $0.083$  & $8$  & $0.093$ \\
$q = 1$   & $3$ & $0.002$  & $5$ & $0.022$  & $7$ & $0.058$  & $8$ & $0.084$  & $9$  & $0.121$ \\
$q = 2$   & $3$ & $0.002$  & $5$ & $0.019$  & $7$ & $0.068$  & $8$ & $0.091$  & $9$  & $0.121$ \\
$q = 3$   & $3$ & $0.002$  & $5$ & $0.018$  & $7$ & $0.070$  & $8$ & $0.095$  & $9$  & $0.125$ \\
$q = 4$   & $3$ & $0.002$  & $5$ & $0.017$  & $7$ & $0.069$  & $8$ & $0.098$  & $10$ & $0.142$ \\
$q = 5$   & $3$ & $0.002$  & $5$ & $0.017$  & $8$ & $0.082$  & $9$ & $0.118$  & $10$ & $0.145$ \\
$q = 6$   & $4$ & $0.005$  & $4$ & $0.010$  & $8$ & $0.092$  & $9$ & $0.125$  & $11$ & $0.181$ \\
\end{tabular} \vspace{2ex}
\caption{Example~\ref{ex:3}: case [a] with $K(x)=10^q$ and two smoothing steps ($m=2$).}\label{table:a_bilinear_alg1_V_m2}
\end{center}
\end{table}
\vspace{-6mm}
\begin{example}\label{ex:4}
In the next set of numerical experiments we consider the same distribution of inclusions of low permeability
as before but this time against the background of a randomly distributed piecewise constant permeability coefficient
as shown in Figure~\ref{fig:islands_random}. 
\end{example}
The results, presented in Tables~\ref{table:b_bilinear_alg1_V_m0} and~\ref{table:b_bilinear_alg1_V_m2}, are even
better than those obtained for the binary distribution in the sense that both the $V$- and $W$-cycle are robust
with respect to the contrast.
\begin{table}[h!]
 \begin{center}
 \begin{tabular}{c| c  c | c  c | c c  | c c | c  c }
 \multicolumn {11}{c}{ASMG V-cycle: bilinear form~\eqref{WH-div-prod},
Algorithm~\ref{algorithm1}} \\
\multicolumn{1}{c}{~} & \multicolumn{2}{ c |}{$\ell=3$} & \multicolumn{2}{c|}{$\ell=4$}
& \multicolumn{2}{c|}{$\ell=5$} & \multicolumn{2}{c|}{$\ell=6$}
& \multicolumn{2}{c}{$\ell=7$} 
\\
\cline{2-11}
& $n_{ASMG}$ & $\rho$ & $n_{ASMG}$ & $\rho$ & $n_{ASMG}$ & $\rho$   &   $n_{ASMG}$ & $\rho$ & $n_{ASMG}$ & $\rho$ \\
\hline 
$q = 0$   & $4$ & $0.007$  & $6$ & $0.027$  & $9$  & $0.102$  & $10$ & $0.156$  & $12$ & $0.210$ \\
$q = 1$   & $4$ & $0.006$  & $6$ & $0.035$  & $9$  & $0.103$  & $11$ & $0.171$  & $13$ & $0.224$ \\
$q = 2$   & $4$ & $0.005$  & $6$ & $0.032$  & $9$  & $0.102$  & $11$ & $0.159$  & $13$ & $0.222$ \\
$q = 3$   & $4$ & $0.006$  & $6$ & $0.042$  & $9$  & $0.110$  & $11$ & $0.174$  & $13$ & $0.229$ \\
$q = 4$   & $4$ & $0.006$  & $7$ & $0.043$  & $9$  & $0.127$  & $11$ & $0.183$  & $13$ & $0.233$ \\
$q = 5$   & $4$ & $0.006$  & $7$ & $0.049$  & $10$ & $0.138$  & $12$ & $0.195$  & $13$ & $0.239$ \\
$q = 6$   & $4$ & $0.006$  & $7$ & $0.056$  & $10$ & $0.149$  & $12$ & $0.207$  & $14$ & $0.252$ \\
\end{tabular} \vspace{2ex}
\caption{Example~\ref{ex:4}: case [b], no smoothing steps ($m=0$).}\label{table:b_bilinear_alg1_V_m0}
 \end{center}
\end{table}
\vspace{-1ex}
\begin{table}[ht!]
 \begin{center}
 \begin{tabular}{c| c  c | c  c | c c  | c c | c  c }
 \multicolumn {11}{c}{ASMG V-cycle: bilinear form~\eqref{WH-div-prod}, Algorithm~\ref{algorithm1}} \\
\multicolumn{1}{c}{~} & \multicolumn{2}{ c |}{$\ell=3$} & \multicolumn{2}{c|}{$\ell=4$}
& \multicolumn{2}{c|}{$\ell=5$} & \multicolumn{2}{c|}{$\ell=6$}
& \multicolumn{2}{c}{$\ell=7$} 
\\
\cline{2-11}
& $n_{ASMG}$ & $\rho$ & $n_{ASMG}$ & $\rho$ & $n_{ASMG}$ & $\rho$   &   $n_{ASMG}$ & $\rho$ & $n_{ASMG}$ & $\rho$ \\
\hline 
$q = 0$   & $4$ & $0.005$  & $5$ & $0.024$  & $6$ & $0.046$  & $8$ & $0.083$  & $8$  & $0.091$ \\
$q = 1$   & $4$ & $0.005$  & $6$ & $0.033$  & $7$ & $0.060$  & $8$ & $0.091$  & $9$  & $0.124$ \\
$q = 2$   & $3$ & $0.002$  & $5$ & $0.023$  & $6$ & $0.045$  & $7$ & $0.069$  & $9$  & $0.121$ \\
$q = 3$   & $3$ & $0.002$  & $5$ & $0.021$  & $6$ & $0.043$  & $7$ & $0.071$  & $8$  & $0.100$ \\
$q = 4$   & $4$ & $0.005$  & $5$ & $0.023$  & $6$ & $0.044$  & $8$ & $0.089$  & $9$  & $0.122$ \\
$q = 5$   & $4$ & $0.005$  & $5$ & $0.024$  & $6$ & $0.045$  & $8$ & $0.090$  & $9$  & $0.125$ \\
$q = 6$   & $4$ & $0.005$  & $6$ & $0.034$  & $6$ & $0.045$  & $8$ & $0.091$  & $10$ & $0.142$ \\
\end{tabular} \vspace{2ex}
\caption{Example~\ref{ex:4}: case [b] with two smoothing steps ($m=2$).}\label{table:b_bilinear_alg1_V_m2}
 \end{center}
\end{table}

\vspace{-1ex}

\begin{table}[ht!]
 \begin{center}
 \begin{tabular}{c| c  c | c  c | c c  | c c | c  c }
 \multicolumn {11}{c}{ASMG W-cycle: bilinear form~\eqref{WH-div-prod}, Algorithm~\ref{algorithm1}} \\
\multicolumn{1}{c}{~} & \multicolumn{2}{ c |}{$\ell=3$} & \multicolumn{2}{c|}{$\ell=4$}
& \multicolumn{2}{c|}{$\ell=5$} & \multicolumn{2}{c|}{$\ell=6$}
& \multicolumn{2}{c}{$\ell=7$} 
\\
\cline{2-11}
& $n_{ASMG}$ & $\rho$ & $n_{ASMG}$ & $\rho$ & $n_{ASMG}$ & $\rho$   &   $n_{ASMG}$ & $\rho$ & $n_{ASMG}$ & $\rho$  \\
\hline 
$q = 0$   & $4$ & $0.005$  & $4$ & $0.007$  & $4$ & $0.006$  & $4$ & $0.005$  & $4$ & $0.005$ \\
$q = 1$   & $4$ & $0.006$  & $4$ & $0.007$  & $4$ & $0.007$  & $4$ & $0.006$  & $4$ & $0.005$ \\
$q = 2$   & $4$ & $0.004$  & $4$ & $0.009$  & $5$ & $0.016$  & $4$ & $0.007$  & $4$ & $0.006$ \\
$q = 3$   & $4$ & $0.005$  & $5$ & $0.015$  & $5$ & $0.015$  & $4$ & $0.009$  & $4$ & $0.006$ \\
$q = 4$   & $4$ & $0.005$  & $5$ & $0.016$  & $5$ & $0.016$  & $4$ & $0.009$  & $4$ & $0.008$ \\
$q = 5$   & $4$ & $0.005$  & $5$ & $0.018$  & $5$ & $0.015$  & $4$ & $0.009$  & $4$ & $0.008$ \\
$q = 6$   & $4$ & $0.005$  & $5$ & $0.019$  & $5$ & $0.015$  & $4$ & $0.008$  & $4$ & $0.007$ \\
\end{tabular} \vspace{2ex}
\caption{Example~\ref{ex:4}: case [b] with one smoothing step ($m=1$).}\label{table:b_bilinear_alg1_W_m1}
 \end{center}
\end{table}

\medskip
\begin{example} \label{ex:5}
The last set of experiments in the first category is devoted to test case~[c] where,
similarly to Example~\ref{ex:1}, we examine the performance of the preconditioner
for a bilinear form~\eqref{WH-div-prod}. Here, we compare the ASMG preconditioners
for three different coefficient distributions, namely slices 44, 54, and 74 of the SPE10 benchmark
problem. In this example the finest mesh is always composed of $256\times 256$ elements,
meaning that changing the number of levels $\ell$ refers to a different size of the
coarse-grid problem.
\end{example}
Tables~\ref{table:c44_bilinear_VW}--\ref{table:c74_bilinear_VW} report the number
of outer iterations $n_{ASMG}$ and the maximum number of inner iterations $n_i$ needed
to reduce the residual  with the matrix $R\widetilde{D}R^T$ by a  factor of~$10^6$.

\begin{table}[ht!]
 \begin{center}
 \begin{tabular}{c| c  c  c  | c c c | c c c | c c c }
 \multicolumn {13}{c}{ASMG V-cycle and W-cycle: bilinear form~\eqref{WH-div-prod}} \\
\multicolumn{1}{c}{~} & \multicolumn{6}{c|}{V-cycle} & \multicolumn{6}{c}{W-cycle}
\\
\cline{2-13}
 & \multicolumn{3}{ c |}{$m=0$} & \multicolumn{3}{c|}{$m=1$}
& \multicolumn{3}{c|}{$m=0$} & \multicolumn{3}{c}{$m=1$}\\
\cline{2-13}
& $n_{ASMG}$ & $\rho$ & $n_{i}$ & $n_{ASMG}$ & $\rho$ & $n_{i}$   &  
   $n_{ASMG}$ & $\rho$ & $n_{i}$ & $n_{ASMG}$ & $\rho$ & $n_{i}$  \\
\hline 
$\ell = 3$   & $8$  & $0.080$ & $4$ & $7$  & $0.064$ & $5$    &   $5$ & $0.019$ & $6$ & $5$ & $0.014$ & $5$  \\ 
$\ell = 4$   & $10$ & $0.157$ & $6$ & $9$  & $0.122$ & $6$    &   $5$ & $0.019$ & $6$ & $5$ & $0.014$ & $5$ \\ 
$\ell = 5$   & $12$ & $0.209$ & $6$ & $10$ & $0.154$ & $6$    & $5$ & $0.019$ & $6$ & $5$ & $0.014$ & $5$ \\
$\ell = 6$   & $13$ & $0.239$ & $6$ & $11$ & $0.179$ & $6$    & $5$ & $0.019$ & $6$ & $5$ & $0.014$ & $5$ \\ 
$\ell = 7$   & $13$ & $0.239$ & $6$ & $11$ & $0.179$ & $6$    & $5$ & $0.019$ & $6$ & $5$ & $0.014$ & $5$\\ 
\end{tabular} \vspace{2ex}
\caption{Example~\ref{ex:5}: case [c] - slice 44 of the SPE10 benchmark.}\label{table:c44_bilinear_VW}
\end{center}
\end{table}


\vspace{-1ex}

\begin{table}[ht!]
 \begin{center}
 \begin{tabular}{c| c  c  c  | c c c | c c c | c c c }
 \multicolumn {13}{c}{ASMG V-cycle and W-cycle: bilinear form~\eqref{WH-div-prod}} \\
\multicolumn{1}{c}{~} & \multicolumn{6}{c|}{V-cycle} & \multicolumn{6}{c}{W-cycle}
\\
\cline{2-13}
 & \multicolumn{3}{ c |}{$m=0$} & \multicolumn{3}{c|}{$m=1$}
& \multicolumn{3}{c|}{$m=0$} & \multicolumn{3}{c}{$m=1$}\\
\cline{2-13}
& $n_{ASMG}$ & $\rho$ & $n_{i}$ & $n_{ASMG}$ & $\rho$ & $n_{i}$   &  
   $n_{ASMG}$ & $\rho$ & $n_{i}$ & $n_{ASMG}$ & $\rho$ & $n_{i}$  \\
\hline 
$\ell = 3$   & $7$  & $0.070$ & $4$ & $7$  & $0.059$ & $4$    &  $5$ & $0.016$ & $4$ & $5$ & $0.013$ & $4$ \\ 
$\ell = 4$   & $10$ & $0.156$ & $5$ & $9$  & $0.122$ & $6$    &  $5$ & $0.017$ & $6$ & $5$ & $0.013$ & $5$ \\
$\ell = 5$   & $13$ & $0.236$ & $5$ & $11$ & $0.173$ & $6$    &   $5$ & $0.018$ & $6$ & $5$ & $0.013$ & $6$ \\
$\ell = 6$   & $14$ & $0.253$ & $5$ & $11$ & $0.183$ & $6$    &  $5$ & $0.018$ & $6$ & $5$ & $0.013$ & $6$ \\
$\ell = 7$   & $14$ & $0.253$ & $6$ & $11$ & $0.183$ & $6$    &  $5$ & $0.018$ & $6$ & $5$ & $0.013$ & $6$ \\
\end{tabular} \vspace{2ex}
\caption{Example~\ref{ex:5}: case [c] - slice 54 of the SPE10 benchmark.}\label{table:c54_bilinear_VW}
 \end{center}
\end{table}

\vspace{-1ex}

\begin{table}[ht!]
 \begin{center}
 \begin{tabular}{c| c  c  c  | c c c | c c c | c c c }
 \multicolumn {13}{c}{ASMG V-cycle and W-cycle: bilinear form~\eqref{WH-div-prod}} \\
\multicolumn{1}{c}{~} & \multicolumn{6}{c|}{V-cycle} & \multicolumn{6}{c}{W-cycle}
\\
\cline{2-13}
 & \multicolumn{3}{ c |}{$m=0$} & \multicolumn{3}{c|}{$m=1$}
& \multicolumn{3}{c|}{$m=0$} & \multicolumn{3}{c}{$m=1$}\\
\cline{2-13}
& $n_{ASMG}$ & $\rho$ & $n_{i}$ & $n_{ASMG}$ & $\rho$ & $n_{i}$  &  
    $n_{ASMG}$ & $\rho$ & $n_{i}$ & $n_{ASMG}$ & $\rho$ & $n_{i}$  \\
\hline 
$\ell = 3$   & $8$  & $0.090$ & $4$ & $7$  & $0.070$ & $4$    & $5$ & $0.019$ & $4$ & $5$ & $0.014$ & $4$  \\ 
$\ell = 4$   & $11$ & $0.178$ & $5$ & $10$ & $0.145$ & $5$    &  $5$ & $0.020$ & $5$ & $5$ & $0.015$ & $5$ \\
$\ell = 5$   & $13$ & $0.229$ & $5$ & $11$ & $0.166$ & $6$    & $5$ & $0.020$ & $6$ & $5$ & $0.015$ & $6$ \\
$\ell = 6$   & $13$ & $0.242$ & $6$ & $11$ & $0.180$ & $6$    & $5$ & $0.020$ & $6$ & $5$ & $0.015$ & $6$  \\
$\ell = 7$   & $13$ & $0.242$ & $6$ & $11$ & $0.180$ & $6$    & $5$ & $0.020$ & $6$ & $5$ & $0.015$ & $6$ \\
\end{tabular} \vspace{2ex}
\caption{Example~\ref{ex:5}: case [c] - slice 74 of the SPE10 benchmark.}\label{table:c74_bilinear_VW}
 \end{center}
\end{table}

\medskip
\subsection{Testing of block-diagonal preconditioner for system~\eqref{eq:saddle_point_sys-intro}
within MinRes iteration}\label{sec:saddle_point_sys}
Now we present a number of numerical experiments for solving the mixed finite element
system~\eqref{eq:saddle_point_sys-intro} with a preconditioned MinRes method.
We consider two different examples, firstly, Example~\ref{ex:6} in which the performance
of the block-diagonal preconditioner and its dependence on the accuracy of the inner
solves with W-cycle ASMG preconditioner is evaluated, and secondly, 
Example~\ref{ex:7}.
testing the scalability of the MinRes iteration, again using a W-cycle ASMG preconditioner
with one smoothing step for the inner iterations.
\begin{example}\label{ex:6}
  Here we apply the MinRes iteration to
  solve~\eqref{eq:saddle_point_sys-intro} for test case [c]. The hierarchy
  of meshes is the same as in Example~\ref{ex:2}.
An ASMG W-cycle
  based on Algorithm~\ref{algorithm1} with one smoothing step has been
  used as a preconditioner on the $\RT{0}$
  space. Table~\ref{table:c44_saddle_W_m1} shows the number of MinRes
  iterations denoted by $n_{MinRes}$, the
  maximum number of ASMG iterations $n_{ASMG}$ needed to achieve an
  ASMG residual reduction by $\varpi$.
\end{example}
%
\begin{table}[ht!]
 \begin{center}
 \begin{tabular}{c |  c  c  | c c | c c }
 \multicolumn {7}{c}{MinRes iteration: saddle point system~\eqref{eq:saddle_point_sys-intro}} \\
\multicolumn{1}{c}{~} & \multicolumn{2}{c|}{$\varpi=10^6$} & \multicolumn{2}{c|}{$\varpi=10^8$}
& \multicolumn{2}{c}{$\varpi=10^{10}$}
 \\
\cline{2-7}
&  $n_{MinRes}$ & $n_{ASMG}$ 
& $n_{MinRes}$ & $n_{ASMG}$  
& $n_{MinRes}$ & $n_{ASMG}$  
\\
\hline 
$\ell = 3$ & $24$ & $4$ 
& $17$ & $6$ 
& $15$ & $8$ 
\\ 
$\ell = 4$ & $15$ & $5$ 
& $13$ & $6$ 
& $13$ & $8$ 
\\
$\ell = 5$ & $21$ & $5$ 
& $17$ & $6$ 
& $15$ & $8$ 
\\
$\ell = 6$ & $22$ & $5$ 
& $17$ & $6$ 
& $15$ & $8$ 
\\
$\ell = 7$ & $22$ & $5$ 
& $17$ & $6$ 
& $15$ & $8$ 
\\
\end{tabular} \vspace{2ex}
\caption{Example~\ref{ex:6}: case [c] - slice 44 of the SPE10 benchmark. The hierarchy
of meshes is the same as in Example~\ref{ex:5}.}\label{table:c44_saddle_W_m1}
\end{center}
\end{table}

\begin{example}\label{ex:7}
In this set of experiments the MinRes iteration has been used to
solve~\eqref{eq:saddle_point_sys-intro} for test case~[c] for the same hierarchy of
meshes as in Example~\ref{ex:1}. An ASMG W-cycle based on Algorithm~\ref{algorithm1} with
one smoothing step has been used as a preconditioner on the $\RT{0}$ space for a residual
reduction by $10^8$. Table~\ref{table:c44_saddle_W_m1_m} shows the number of MinRes
iterations $n_{MinRes}$,
the maximum number of inner ASMG iterations $n_{ASMG}$ per outer MinRes iteration,
and the number of DOF. Note that so long as the product $n_{MinRes}n_{ASMG}$ is constant,
the total number of arithmetic operations required to achieve any prescribed
accuracy is proportional to the number of DOF.
\end{example}

\begin{table}[ht!]
 \begin{center}
 \begin{tabular}{c|r|cc|cc}
 \multicolumn{6}{c}{ MinRes iteration: saddle point system~\eqref{eq:saddle_point_sys-intro}} \\
 & &          & {zero r.h.s.}&  & {nonzero r.h.s.}   \\
\hline
 & {DOF}&  $n_{MinRes}$ & $n_{ASMG}$ &  $n_{MinRes}$ & $n_{ASMG}$  \\
\hline 
$\ell = 4$ &  $3,136$ & $13$ & $5$ & $13$   &  $5$    \\
$\ell = 5$ & $12,416$ & $13$ & $6$ & $14$  &  $6$    \\ 
$\ell = 6$ & $49,408$ & $15$ & $6$ & $17$ & $6$ 
\\
$\ell = 7$ & $197,120$ & $17$ & $6$ & $17$  & $6$ 
\\
$\ell = 8$ & $787,456$  & $17$ & $6$ & $18$ & $6$ 
\\
\end{tabular} \vspace{2ex}
\caption{ Example~\ref{ex:7}: case [c] - slice 44 of the SPE10
benchmark. }  
\label{table:c44_saddle_W_m1_m}
\end{center}
\end{table}

\subsection{Comments regarding the numerical experiments and some general conclusions}

The presented numerical results clearly demonstrate the efficiency of the proposed algebraic
multilevel iteration (AMLI)-cycle auxiliary space multigrid (ASMG) preconditioner
for problems with highly varying coefficients
as they typically arise in the mathematical modelling of physical processes in high-contrast
and high-frequency media.

During the first tests we evaluated the quantity 
$\Vert \pi_{\widetilde{D}} \Vert^2_{\widetilde{A}}$,
 which by  Lemma \ref{lem:key}  provides an upper 
bound for the condition number $\kappa(C^{-1}A)$. Then the convergence factor in A-norm of 
the linear V-cycle method was numerically studied.  The above reported results show 
robustness with respect to a highly varying coefficient on multiple length scales.  
They also confirm that the nonlinear V-cycle reduces the residual norm faster than the 
linear V-cycle.

The next group of tests examines the convergence behaviour of the nonlinear ASMG method
for the weighted bilinear form~\eqref{WH-div-prod}.
This is a key point in the presented study.
Cases~[a] and~[b] are designed to represent a typical multiscale geometry with
islands and channels. Although case~[b], a background with a random coefficient,
appears to be more complicated, the impact of the multiscale heterogeneity seems to
be stronger in binary case~[a] where the number of iterations is slightly larger.
However, in both cases we observe a uniformly converging ASMG V-cycle
with $m=2$ and W-cycle ($\nu=2$) with $m=1$. 
Case~[c] (SPE10) is a benchmark problem in the petroleum engineering
community. Here we observe robust and uniform convergence with respect
to the number of levels $\ell$, or, equivalently, mesh-size $h$. Note
that such uniform convergence is recorded for the ASMG V-cycle even
without smoothing iterations (i.e.  $m=0$). 

The results presented in
Table~\ref{table:c44_saddle_W_m1}--\ref{table:c44_saddle_W_m1_m},
confirm the expected optimal convergence rate
of the block-diagonally preconditioned MinRes iteration applied
to the coupled saddle point system~\eqref{eq:saddle_point_sys-intro}.
Results in Table~\ref{table:c44_saddle_W_m1} demonstrate
how the efficiency (in terms of the product $n_{MinRes} n_{ASMG}$) is achieved
for a relative accuracy of $10^{-8}$ of the inner ASMG solver.
Table \ref{table:c44_saddle_W_m1_m} illustrates the scalability of the solver
indicated by an almost constant number of  MinRes and ASMG
iterations since the total computational work 
in terms of fine grid matrix vector multiplications is proportional to
the product $n_{MinRes}n_{ASMG}$.
The case of a non-homogeneous right hand side  provides a promising indicator 
for robustness of the ASMG preconditioner beyond the frame of the presented theoretical
analysis.    

Although not in the scope of this study, we note that the proposed auxiliary space multigrid method would be suitable for implementation
on distributed memory computer architectures.

\section*{Acknowledgments}

The authors express their sincere thanks to the anonymous reviewers who made a number of
critical remarks and raised relevant questions that resulted in a revised and improved version of
the paper which is now shorter, clearer and more precise in the presentation of the main results. 

The authors sincerely thank Dr.~Shaun Lymbery for his contribution to editing the paper.

This work has been partially supported by the Bulgarian NSF Grant DCVP
02/1, FP7 Grant AComIn, and the Austrian NSF Grant P22989. R. Lazarov
has been supported in part by US NSF Grant DMS-1016525, by Award
No. KUS-C1-016-04, made by KAUST. L. Zikatanov has been supported in
part by US NSF Grants DMS-1217142 and DMS-1016525.

\bibliographystyle{abbrv}
\bibliography{thebib3}

\end{document}